\documentclass[reqno]{amsart}

\usepackage{pstricks,amsmath,amsthm,amssymb,amsfonts}
\numberwithin{equation}{section} 
\theoremstyle{plain}
\newtheorem{theo+}           {Theorem}      [section]
\newtheorem{prop+}  [theo+]  {Proposition}
\newtheorem{coro+}  [theo+]  {Corollary}
\newtheorem{lemm+}  [theo+]  {Lemma}
\newtheorem{defi+}  [theo+]  {Definition}

\theoremstyle{definition}
\newtheorem{rema+}  [theo+]  {Remark}
\newtheorem{prob+}  [theo+]  {Problem}
\newtheorem{exam+}  [theo+]  {Example}

\newenvironment{theorem}{\begin{theo+}}{\end{theo+}}
\newenvironment{proposition}{\begin{prop+}}{\end{prop+}}
\newenvironment{corollary}{\begin{coro+}}{\end{coro+}}
\newenvironment{lemma}{\begin{lemm+}}{\end{lemm+}}

\newenvironment{remark}{\begin{rema+}}{\end{rema+}}
\newenvironment{definition}{\begin{defi+}}{\end{defi+}}

\newcommand{\iceR}{\psline{->}(0,0)(.7,0) \psline(1,0)}
\newcommand{\iceL}{\psline{->}(0,0)(-.7,0) \psline(-1,0)}
\newcommand{\iceU}{\psline{->}(0,0)(0,.7) \psline(0,1)}
\newcommand{\iceD}{\psline{->}(0,0)(0,-.7) \psline(0,-1)}

\newcommand{\al}{\alpha}
\newcommand{\be}{\beta}
\newcommand{\ga}{\gamma}

\newcommand{\de}{\delta}
\newcommand{\De}{\Delta}
\newcommand{\ep}{\varepsilon}
\newcommand{\la}{\lambda}

\newcommand{\om}{\omega}

\newcommand{\si}{\sigma}
\newcommand{\tha}{\theta}

\newcommand{\gh}{\mathfrak h}
\newcommand{\gha}{\mathfrak h^\ast}
\newcommand{\mg}{M_{\gh^\ast}}

\newcommand{\D}{D_{\gh}}

\newcommand{\End}{\operatorname{End}}
\newcommand{\id}{\operatorname{id}}

\newcommand{\op}{\operatorname{op}}
\newcommand{\cop}{\operatorname{cop}}
\newcommand{\coop}{\operatorname{coop}}

\newcommand{\ot}{\otimes}
\newcommand{\C}{\mathbb C}
\newcommand{\Cp}{\mathbb C^\times}

\newcommand{\wh}{\widehat}
\newcommand{\wt}{\widetilde}
\newcommand{\E}{ \mathcal E}
\newcommand{\lx}{\langle}
\newcommand{\rx}{\rangle}
\newcommand{\ra}{\rightarrow}

\begin{document}

\baselineskip 18pt
\larger[2]
\title
[Felder's elliptic quantum group]
{Felder's elliptic  quantum group
and elliptic hypergeometric series on the root system $A_n$} 
\author{Hjalmar Rosengren}
\address
{Department of Mathematical Sciences
\\ Chalmers University of Technology and G\"oteborg
 University\\SE-412~96 G\"oteborg, Sweden}
\email{hjalmar@chalmers.se}
\urladdr{http://www.math.chalmers.se/{\textasciitilde}hjalmar}

\thanks{Research  supported by the Swedish Science Research
Council (Vetenskapsr\aa det)}

\begin{abstract}
We introduce a generalization of elliptic $6j$-symbols, which can be interpreted as matrix elements for  intertwiners between  corepresentations of Felder's elliptic quantum group. For special parameter values, they can be expressed in terms of multivariable elliptic hypergeometric series related to the root system $A_n$. As a consequence, we obtain new biorthogonality relations for such series.
\end{abstract}

\maketitle        

\section{Introduction}

In Baxter's  solution of the eight-vertex model, a decisive step is the introduction of the  \emph{eight-vertex-solid-on-solid} (8VSOS) 
\emph{model}, the two models being related by a vertex-face transformation \cite{b}. The $R$-matrix of the 8VSOS model satisfies a modified version of the quantum Yang--Baxter equation known as the  \emph{quantum dynamical Yang--Baxter} (QDYB) \emph{equation}. The QDYB equation includes as a special case  the hexagon relation for $6j$-symbols, which was first found by Wigner around 1940 \cite{w}.

Two currently active research areas arising from  the 8VSOS model
are \emph{dynamical quantum groups} and \emph{elliptic hypergeometric functions}. For surveys, see \cite{es} and \cite[Chapter 11]{gr} or \cite{ss}, respectively. The aim of the present work is to find a new and unexpected connection between these two areas. 

To explain the notion of a dynamical quantum group, we recall  the FRST (Faddeev--Reshetikhin--Sklyanin--Takhtajan) construction \cite{frt}, which  associates a bialgebra (in many cases a Hopf algebra) to any quantum $R$-matrix. Physically interesting quantities  can then be studied by algebraic methods.  It is not obvious how to generalize the FRST construction to \emph{dynamical} $R$-matrices. A  first step was taken by Felder \cite{f}, who, 
rather than defining a dynamical quantum group as an algebra,
defined a category of its   representations. In the case of 
Baxter's dynamical $R$-matrix, pertaining to
the 8VSOS model, this category was studied in detail by Felder and Varchenko \cite{fv}.  Later, Etingof and Varchenko  \cite{ev} gave a solid algebraic 
foundation for dynamical quantum groups as \emph{$\gh$-Hopf algebroids}.
By \emph{Felder's  quantum group}, we  mean the 
 $\gh$-Hopf algebroid obtained by applying the generalized FRST construction of Etingof and Varchenko to Baxter's dynamical $R$-matrix.

In \cite{d}, Date et al.\ applied a process of fusion to the 8VSOS model, thereby obtaining more general dynamical $R$-matrices known as \emph{elliptic $6j$-symbols}.
The classical $6j$-symbols of quantum mechanics can be expressed as  hypergeometric sums, that is, as $\sum_k a_k$ with $a_{k+1}/a_k$ a rational function of $k$. 
Elliptic $6j$-symbols are given by more general sums, with
  $a_{k+1}/a_k$  an elliptic function of $k$. Frenkel and Turaev \cite{ft} realized that this was
the  first example of a completely new class of special functions, \emph{elliptic hypergeometric functions}.

Another relatively recent  class of special functions are 
\emph{hypergeometric functions on root systems}, which
 first appeared in the context of $6j$-symbols of unitary groups  \cite{ajj,ccb}. Similarly as for elliptic hypergeometric series, examples
appeared in the physics literature long before their nature as 
generalized hypergeometric series was emphasized, in this case by Holman, Biedenharn and Louck \cite{hbl}. Recently, many authors have considered \emph{elliptic} hypergeometric functions on root systems,   see e.g.\ \cite{ds,ds2,kan,ra,ras,rr,rt,rsm,s1,s2,s3,sw}. (To be  precise, we think here mainly of integrals and sums of   type I or Dixon-type, as opposed to type II or Selberg-type.) 
A major recent development is the appearance of  elliptic hypergeometric integrals on root systems in the context of supersymmetric quantum field theories \cite{do,sv}, which has 
led to many new conjectures. In spite of the origin of hypergeometric functions on root systems in the representation theory of unitary groups,  little has been written about connections with Lie or quantum groups, and in the elliptic case nothing seems to be known before the present work.

In view of their common origin in Baxter's 8VSOS model, one would expect  direct relations between Felder's  quantum group and elliptic hypergeometric functions. Such relations were obtained in \cite{knr}, 
and in a somewhat different way in \cite{kn}.
Felder's  quantum group has a one-parameter family of 
$(N+1)$-dimensional irreducible corepresentations $V_N(z)$. In \cite{kn}, Koelink and van Norden considered pairings of the form $\lx M_{st}^M(w),M_{uv}^N(z)\rx$, where
$M_{st}^M(w)$ denotes a matrix element of   $V_M(w)$, and $\lx\cdot,\cdot\rx$ is the \emph{cobraiding} on the quantum group.
These pairings are matrix elements of the natural intertwiner
\begin{equation}\label{vvi}V_N(z)\wh\ot V_M(w) \ra V_M(w)\wh\ot V_N(z)\end{equation}
 ($\wh\ot$ is a modified  tensor product appropriate for dealing with Hopf algebroids). 
It was shown that such pairings can be identified with the
elliptic $6j$-symbols of \cite{d}.

In the present paper, we consider
 rather than \eqref{vvi} the  intertwiner
\begin{multline}\label{gi} \left(V_1(z_1)\wh\ot\dots\wh\ot V_1(z_N)\right)\wh\ot\left(V_1(w_1)\wh\ot\dots\wh\ot V_1(w_M)\right) \\
 \ra \left(V_1(w_1)\wh\ot\dots\wh \ot V_1(w_M)\right)\wh\ot \left(V_1(z_1)\wh\ot\dots\wh\ot V_1(z_N)\right). 
\end{multline}
In a natural  basis of pure tensors, the matrix elements of this intertwiner are   partition functions of the 8VSOS model with  fixed boundary conditions. However, we will consider the same intertwiner in a \emph{different} basis, when the matrix elements can be considered as generalized $6j$-symbols.

The comodule $V_N(z)$  is a quotient of 
$$V_1(z)\ot V_1(qz)\ot\dotsm\ot V_1(q^{N-1}z)$$ 
($q$ is a parameter of the $R$-matrix). Accordingly, when  $w_j= q^{j-1}\om$ and $z_j=q^{j-1}\zeta$, our generalized $6j$-symbols reduce to the elliptic $6j$-symbols of \cite{d}. 
 One of our main results is that, if only $w$ is  specialized to a geometric progression, then the generalized $6j$-symbols can be expressed in terms of elliptic hypergeometric series related to the root system $A_n$, of  the type studied in  \cite{kan,rr,rt,s1,s2}. 
This connection allows us to obtain new biorthogonality relations for such series.
Surprisingly, although degenerate cases  first appeared in the study of  $\mathrm{SU}(n)$, we obtain more general functions using an elliptic deformation of mere $\mathrm{SU}(2)$. 

The original motivation for the present study was not the link to hypergeometric series on root systems, which in fact came as a surprise. Rather, it is part of an ongoing project to develop harmonic analysis on  dynamical quantum groups, with Felder's quantum group as the main example.
In particular, we believe that some of our findings will be useful for constructing  a Haar functional on  Felder's quantum group, and for obtaining a more concrete version of the construction of solutions to the $q$-Knizhnik--Zamolodchikov--Bernard equation due to Varchenko and co-workers \cite{ftv1,ftv2,fv2,mv}.
Finally, in view of the recent appearance of elliptic hypergeometric integrals in quantum field theory mentioned above, one may speculate that the biorthogonal system of Theorem \ref{bp}, or related systems with continuous biorthogonality measures, has a role to play in that context.

The plan of the paper is as follows. In \S \ref{dqs}, we give preliminaries on dynamical quantum groups. This includes some new definitions and results, in particular on unitary cobraidings on $\gh$-Hopf algebroids. In \S \ref{fgs}, we 
recall the definition of Felder's quantum group, and collect some
 elementary though useful results on its cobraiding. In particular, Corollary \ref{rl} is a key result that should have some independent interest.
In \S \ref{ecs}, we introduce the corepresentations and bases
 that we will use. In \S \ref{ews}, we discuss a function that  appears as a building block of our generalized $6j$-symbols. In terms of the 8VSOS model, it is  the domain wall partition function; it can also be identified with  elliptic weight functions of Tarasov and Varchenko \cite{tv}. 
In \S \ref{sss}, we introduce  generalized $6j$-symbols and study their main properties. In  particular, in Theorems \ref{mcmt} and \ref{rat} we give 
explicit expressions for these symbols.  Although these formulas may seem complicated, they are  natural extensions   of Racah's expression for the classical $6j$-symbol as a ${}_4F_3$ hypergeometric sum. In \S \ref{hss}, we consider the specialization of generalized $6j$-symbols that leads to
elliptic hypergeometric series on the root system $A_n$. As an application, in Theorem \ref{bp} we obtain an explicit biorthogonality relation for such series. 
Finally, an Appendix contains  generalities on  unitary  symmetries of cobraided $\gh$-Hopf algebroids.

{\bf Acknowledgements:} I  thank Jonas Hartwig 
for many discussions, and
Vitaly Tarasov for illuminating correspondence.

\section{Dynamical quantum groups}
\label{dqs}

This Section contains preliminaries on $\gh$-Hopf algebroids, most of which can be found in \cite{ev,kn,kr,r}.

\subsection{$\gh$-Hopf algebroids}

Let $\gha$ be a finite-dimensional complex vector space. 
The notation is motivated by examples where 
 $\gha$ is the dual of the Cartan subalgebra of a semisimple Lie algebra. In the case of interest to us, $\gha=\C$.

We denote by $\mg$ the field of meromorphic functions on $\gha$, and by $T_\al$, $\al\in\gha$, the shift operators $(T_\al f)(\la)=f(\la+\al)$ acting on $\mg$. Moreover,  $\D$ denotes the algebra of
 difference operators $\sum_i f_i T_{\beta_i}$, $f_i\in\mg$, $\beta_i\in\gha$, acting on $\mg$.

An \emph{$\gh$-algebra} is a complex associative algebra with $1$, which is bigraded over $\gha$, that is, $A=\bigoplus_{\al,\be\in\gha}A_{\al\be}$, with $A_{\al\be}A_{\ga\de}\subseteq A_{\al+\ga,\be+\de}$. Moreover, there are two algebra embeddings $\mu_l,\ \mu_r:\mg\ra A_{00}$ (the left and right \emph{moment maps}), such that 
$$\mu_l(f)a=a\mu_l(T_\al f),\qquad \mu_r(f)a=a\mu_r(T_\be f),\qquad a\in A_{\al\be},\quad f\in\mg.$$     
A morphism of $\gh$-algebras is an algebra homomorphism which preserves the bigrading and moment maps.

When $A$ and $B$ are $\gh$-algebras,  $A\wt\ot B$  denotes the quotient of
$\bigoplus_{\al,\be,\ga\in\gha}A_{\al\ga}\ot B_{\ga\be}$ 
by the relations
$\mu_r(f)a\ot b=a\ot\mu_l(f)b$.
The multiplication $(a\ot b)(c\ot d)=ac\ot bd$, the bigrading 
$A_{\al\ga}\wt\ot B_{\ga\be}\subseteq(A\wt\ot B)_{\al\be}$
and the moment maps
$$\mu_l(f)(a\ot b)=\mu_l(f)a\ot b,\qquad \mu_r(f)(a\ot b)=a\ot \mu_r(f)b $$
make $A\wt\ot B$ an $\gh$-algebra.

The bigrading $fT_{-\be}\in (\D)_{\be\be}$ and the moment maps 
$\mu_l(f)=\mu_r(f)=fT_0$ equip 
  $\D$ with the structure of an  $\gh$-algebra. 
It provides a unit object for the tensor product $\wt\ot$, namely,
\begin{equation}\label{ci}x\simeq x\ot T_{-\be}\simeq T_{-\al}\ot x,\quad x\in A_{\al\be}, \end{equation}
  define $\gh$-algebra isomorphisms $A\simeq A\wt\ot \D\simeq \D\wt\ot A$.

An \emph{$\gh$-bialgebroid} is an $\gh$-algebra $A$ equipped with two $\gh$-algebra morphisms $\Delta:\,A\ra A\wt\ot A$ (the \emph{coproduct}) and $\ep:\,A\ra \D$ (the \emph{counit}), such that 
$$(\Delta\otimes\id)\circ\De=(\id\otimes\,\Delta)\circ\De $$
and, under the identifications \eqref{ci}, 
\begin{equation}\label{cucf}(\ep\otimes\id)\circ\De=(\id\ot\,\ep)\circ\De=\id.
\end{equation}

An \emph{$\gh$-Hopf algebroid} is an $\gh$-bialgebroid equipped with a $\C$-linear map $S:\,A\ra A$ (the \emph{antipode}), such that
$S(A_{\al\be})\subseteq A_{-\be,-\al}$,
$S(\mu_r(f))=\mu_l(f)$, 
$S(\mu_l(f))=\mu_r(f)$,
$S(ab)=S(b)S(a)$, $S(1)=1$,
\begin{equation}\label{aprop}\De\circ S=\sigma\circ(S\otimes S)\circ \De,\qquad \ep\circ S=S^{D_{\gh}}\circ\ep,\end{equation}
\begin{equation}\label{ics}m\circ(\id\otimes\, S)\circ\De(a)=\mu_l(\ep(a)1),\end{equation}
$$ m\circ(S\otimes \id)\circ\De(a)=\mu_r(T_\al(\ep(a)1)),
\qquad a\in A_{\al\be},$$
where $m$ denotes multiplication, $\sigma(a\otimes b)=b\otimes a$, and where
 $S^{D_{\gh}}$ is  the antiautomorphism of $\D$ defined by 
\begin{equation}\label{sdd}S^{\D}(f)=f,\qquad S^{\D}(T_\al)=T_{-\al}.\end{equation}
These axioms are far from independent, see \cite{kr}.

We will  use Sweedler's notation
$$\sum_{(a)} a'\ot \dotsm\ot a^{(n)}$$
for the iterated coproduct of $a$. Then,
 \eqref{cucf} may be written
\begin{equation}\label{cua}x=\sum_{(x)}\mu_l(\ep(x')1)x''=\sum_{(x)}\mu_r(\ep(x'')1)x'. \end{equation}
Moreover, 
assuming that each $a^{(i)}$ belongs to some bigraded component $A_{\alpha_i\beta_i}$, we write $\om_{i,i+1}(a)=\beta_i=\al_{i+1}$. For instance,
 if $a\in A_{\al\be}$, then
$$(\id\ot\,\Delta)\Delta(a)=(\Delta\ot\id)\Delta(a)=\sum_{(a)}a'\ot a''\ot a^{(3)},$$
where $a'\in A_{\al,\om_{12}(a)}$, $a''\in A_{\om_{12}(a),\om_{23}(a)}$, $a^{(3)}\in A_{\om_{23}(a),\be}$.

\subsection{Corepresentations}

An \emph{$\gh$-space} is  an $\gha$-graded vector space over $\mg$, $V=\bigoplus_{\al\in \gha}V_{\al}$. A morphism of $\gh$-spaces is an $\mg$-linear and grade-preserving map.

If $A$ is an $\gh$-bialgebroid and $V$ an {$\gh$-space}, $A\wt\ot V$ denotes the quotient of
$$\bigoplus_{\al,\be\in\gha}A_{\al\be}\ot V_\be$$ 
by the  relations
$\mu_r(f)a\ot v=a\ot fv$.
The grading $A_{\al\be}\wt\ot V_\be\subseteq(A\wt\ot V)_\al$ and the $\mg$-linear structure $f(a\ot v)=\mu_l(f)a\ot v$ make $A\wt\ot V$ an $\gh$-space.
The identification  
\begin{equation}\label{dvi}f\,T_{-\al}\otimes v\simeq fv, \qquad v\in V_\al,\end{equation}
gives an $\gh$-space isomorphism
$D_{\gh}\widetilde\otimes V\simeq V$.

A \emph{corepresentation} of $A$ on $V$ is an $\gh$-space morphism $\pi:\,V\ra A\wt\ot V$ such that
\begin{equation}\label{crd}(\De\otimes\id)\circ\pi=(\id\otimes\,\pi)\circ\pi,\qquad (\ep\otimes\id)\circ\pi=\id,\end{equation}
using \eqref{dvi} in the second equality.  If $(e_x)_{x\in X}$ is a homogeneous basis for $V$ over $\mg$, $e_x\in V_{\om(x)}$, then one can introduce matrix elements $t_{xy}\in A_{\om(x)\om(y)}$ by
\begin{equation}\label{med}\pi(e_x)=\sum_{y\in X}t_{xy}\ot e_y. \end{equation}
In terms of matrix elements, \eqref{crd} takes the form
\begin{equation}\label{crm}\De(t_{xy})=\sum_{z\in X}t_{xz}\ot t_{zy},\qquad \ep(t_{xy})=\de_{xy}T_{-\om(x)}.\end{equation}

Given two $\gh$-spaces $V$ and $W$, their tensor product $V\wh\ot W$ is defined as  $V\ot W$ modulo the relations
$$T_{-\al}(f)v\ot w=v\ot fw,\qquad v\in V_\al,\quad f\in\mg. $$
It is an $\gh$-space with  $V_\al\wh\ot W_\be\subseteq (V\wh\ot W)_{\al+\be}$ and $f(v\ot w)=fv\ot w$. If $V$ and $W$ are corepresentation spaces, then so is $V\wh\ot W$ under 
$$\pi_{V\wh\ot W}=(m\ot\id)\circ\si_{23}\circ(\pi_V\ot \pi_W), $$ 
where $\si_{23}(a\ot v\ot b\ot w)=a\ot b\ot v\ot w$. 
Equivalently, in terms of matrix elements,
\begin{equation}\label{tpme}\pi_{V\wh\ot W}(e_x^V\ot e_y^W)=\sum_{ab}t_{xa}^V t_{yb}^W\ot e_a^V\ot e_b^W.
 \end{equation}

\subsection{FRST construction}
\label{fcs}

Let $X$ be a finite  index set and $\om:\,X\ra \gha$ an arbitrary  function. Let $R=(R^{bd}_{ac})_{a,b,c,d\in X}$ be a matrix, whose elements $R^{bd}_{ac}=R^{bd}_{ac}(\lambda,z)$  are meromorphic functions of $(\lambda,z)\in\gha\times\Cp$, where
$\Cp=\mathbb C\setminus\{0\}$. We refer to $\lambda$ as the
 \emph{dynamical parameter} and $z$ as the \emph{spectral parameter}. Moreover, we require that 
\begin{equation}\label{hir}R^{bd}_{ac}\neq 0\quad 
\Longrightarrow\quad
\om(a)+\om(c)=\om(b)+\om(d). \end{equation}
To any such $R$ one may associate an $\gh$-bialgebroid $A$. 
As a complex algebra, it is generated by two copies of $\mg$, whose elements we write as $f(\la)$ and $f(\mu)$, respectively, together with generators $(L_{xy}(z))_{x,y\in X, z\in\Cp}$. The defining relations are
$$f(\la)L_{xy}=L_{xy}f(\la+\om(x)),
\qquad
 f(\mu)L_{xy}=L_{xy}f(\mu+\om(y)),$$
 $$f(\la)g(\mu)=g(\mu)f(\la),$$
\begin{multline}\label{rll}
\sum_{xy}\lim_{t\rightarrow z/w}(t-z/w)^N R_{ac}^{xy}(\la,t)L_{xb}(z)L_{yd}(w)\\
=\sum_{xy}\lim_{t\rightarrow z/w}(t-z/w)^N R_{xy}^{bd}(\mu,t)L_{cy}(w)L_{ax}(z),
\end{multline}
for $N\in\mathbb Z$ such that all the limits exist.
If each $R^{bd}_{ac}(\lambda,z)$ is holomorphic in $z$, then 
\eqref{rll} reduces to 
$$\sum_{xy}R_{ac}^{xy}(\la,z/w)L_{xb}(z)L_{yd}(w)=
\sum_{xy}R_{xy}^{bd}(\mu,z/w)L_{cy}(w)L_{ax}(z).$$

The bigrading $f(\la),\ f(\mu)\in A_{00}$, $L_{xy}\in A_{\om(x)\om(y)}$, the moment maps $\mu_l(f)=f(\la)$, $\mu_r(f)=f(\mu)$, the coproduct
$$\De(L_{ab}(z))=\sum_{x\in X}L_{ax}(z)\otimes L_{xb}(z),
\qquad \De(f(\la))=f(\la)\otimes 1,\qquad \De(f(\mu))=1\otimes f(\mu)$$
and the counit
$$\ep(L_{ab}(z))=\de_{ab}\,T_{-\om(a)},
\qquad \ep(f(\la))=\ep(f(\mu))=f$$
make $A$ an $\gh$-bialgebroid.

\subsection{QDYB equation}

The FRST construction is of particular interest when
$R$  is a \emph{dynamical $R$-matrix}, meaning that
\begin{multline}\label{dyx}\sum_{xyz} R^{xy}_{de}(\la-\om(f),z_1/z_2)
\,R^{az}_{xf}(\la,z_1/z_3)\,R^{bc}_{yz}(\la-\om(a),z_2/z_3)\\
=\sum_{xyz}R^{yz}_{ef}(\la,z_2/z_3) \,R^{xc}_{dz}(\la-\om(y),z_1/z_3)\,R^{ab}_{xy}(\la,z_1/z_2).
\end{multline}
A dynamical $R$-matrix satisfying
\begin{equation}\label{ur}\sum_{xy}R^{xy}_{ab}(\la,z_1/z_2)\,R_{yx}^{dc}(\la,z_2/z_1)=\delta_{ac}\delta_{bd}\end{equation}
is called \emph{unitary}.

Let $V$ be the complex vector space with basis  $(v_x)_{x\in X}$. 
We identify $R$ with the map
$(\gha\times\mathbb C^\times)\rightarrow\End_\gh(V\ot V)$ given by
$$R(\la,z)(v_x\ot v_y)=\sum_{ab}R_{ab}^{xy}(\la,z)v_a\ot v_b $$
(the subscript in $\End_\gh$ refers to the condition \eqref{hir}).
Then, \eqref{dyx} and \eqref{ur} can be written in coordinate-free form as
\begin{multline*}R^{12}(\la-h^{(3)},z_1/z_2)R^{13}(\lambda,z_1/z_3)R^{23}(\lambda-h^{(1)},z_2/z_3)\\
=R^{23}(\lambda,z_2/z_3)R^{13}(\lambda-h^{(2)},z_1/z_3)R^{12}(\la,z_1/z_2), 
\end{multline*}
$$R^{12}(\la,z_1/z_2)R^{21}(\lambda,z_2/z_1)=\id, $$
where the notation is explained by the example
$$R^{12}(\la-h^{(3)},z_1/z_2)(v_a\ot v_b\ot v_c)=
\sum_{xy}R_{xy}^{ab}(\la-\om(c),z_1/z_2)(v_x\ot v_y\ot v_c).$$

More generally, one may consider the equations
$$\mathcal R^{12}_{UV}(\la-h^{(3)})\mathcal R^{13}_{UW}(\lambda)\mathcal R^{23}_{VW}(\lambda-h^{(1)})
=\mathcal R^{23}_{VW}(\lambda)\mathcal R^{13}_{UW}(\lambda-h^{(2)})\mathcal R^{12}_{UV}(\la),$$
$$\mathcal R^{12}_{UV}(\la)\mathcal R^{21}_{VU}(\lambda)=\id,$$
for operators $\mathcal R_{UV}:\,\gha\rightarrow \End_{\gh}(U\ot V)$, where $U$, $V$ and $W$ are $\gha$-graded complex vector spaces.
Equivalently, if
$$\mathcal R_{UV}(\la)(u_x\ot v_y)=\sum_{ab}\mathcal R^{xy}_{ab}(\la;U,V)(u_a\ot v_b), $$
one has
\begin{multline}\label{gdyc}\sum_{xyz}\mathcal  R^{xy}_{de}(\la-\om(f);U,V)
\,\mathcal R^{az}_{xf}(\la;U,W)\,\mathcal R^{bc}_{yz}(\la-\om(a);V,W)\\
=\sum_{xyz}\mathcal R^{yz}_{ef}(\la;V,W) \,\mathcal R^{xc}_{dz}(\la-\om(y);U,W)\,\mathcal R^{ab}_{xy}(\la;U,V),
\end{multline}
\begin{equation}\label{gduc}\sum_{xy}\mathcal R^{xy}_{ab}(\la;U,V)\,\mathcal R_{yx}^{dc}(\la;V,U)=\delta_{ac}\delta_{bd}.\end{equation}

The QDYB equation is equivalent to the star-triangle relation for certain  generalized ice models, see \cite{b} for the case $\gha=\mathbb C$ and \cite{jkmo,jmo} for examples involving higher rank Lie algebras. 
 In the case when the spectral parameter is absent, it is  the hexagon relation for $6j$-symbols, in the case  $\gha=\mathbb C$ going back to Wigner \cite{w}.

\subsection{Cobraidings}
\label{cbss}

For the following definition, see \cite[Def.\ 3.16]{r}.

\begin{definition}
A \emph{cobraiding} on an  $\gh$-bialgebroid
 $A$ is a
$\C$-bilinear map $\langle\cdot,\cdot\rangle:\,A\times A\rightarrow D_{\gh}$
such that, for any $a,b,c\in A$ and $f\in\mg$,
$$\langle A_{\al\be},A_{\ga\de}\rangle\subseteq
(D_\gh)_{\al+\ga,\be+\de},$$
\begin{subequations}
\label{mrp}
\begin{align}\label{mrpa}\langle\mu_r(f)a,b\rangle&=\langle a,\mu_l(f)b\rangle
=f\circ\langle a,b\rangle,\\
\label{mrpb}\langle a\mu_l(f),b\rangle&=\langle a,b\mu_r(f)\rangle
=\langle a,b\rangle\circ f,\end{align}
\end{subequations}
\begin{subequations}\label{ppp}
\begin{align}\label{pppa}
\langle ab,c\rangle&=\sum_{(c)}\langle a,c'\rangle
\,T_{\om_{12}(c)}\,\langle b,c''\rangle,\\
\label{pppb}\langle a,bc\rangle&=\sum_{(a)}\langle a'',b\rangle
\,T_{\om_{12}(a)}\,\langle a',c\rangle,
\end{align}
\end{subequations}
\begin{equation}\label{cbe}\langle a,1\rangle=\langle 1,a\rangle=\ep(a),
\end{equation}
\begin{equation}
\label{cba}\sum_{(a)(b)}\mu_l(\lx a',b'\rx 1)\,a''b''
=\sum_{(a)(b)}\mu_r(\lx a'',b''\rx 1)\,b'a'.
\end{equation}
\end{definition}

In particular,
\begin{equation}\label{cbv}\langle A_{\al\be},A_{\ga\de}\rangle\neq 0
\quad\Longrightarrow\quad \al+\ga=\be+\de. \end{equation}

The following definition  is motivated by Proposition \ref{up} below.

\begin{definition}
A cobraiding is called \emph{unitary} if 
\begin{equation}\label{ud}\ep(ab)=\sum_{(a)(b)}\lx b',a'\rx T_{\om_{12}(a)+\om_{12}(b)}\lx a'',b''\rx. \end{equation}
\end{definition}

To verify that  a cobraiding is unitary, the following facts are useful.

\begin{lemma}\label{ugl}
 The equation \eqref{ud} always holds  if  $a$ or $b$ is in $\mu_l(\mg)\mu_r(\mg)$. Moreover,
 if \eqref{ud} holds for the pairs $(a_1,b)$ and $(a_2,b)$, then it holds for $(a_1a_2,b)$, and if
 \eqref{ud} holds for the pairs $(a,b_1)$ and $(a,b_2)$, then it holds for $(a,b_1b_2)$.
\end{lemma}

\begin{proof}
We will only prove the last assertion, the second one being similar and the first one  straight-forward. Thus,
we must prove that
$$\ep(abc)=\sum_{(a)(b)(c)}\lx b'c',a'\rx T_{\om_{12}(a)+\om_{12}(b)+\om_{12}(c)}
\lx a'',b''c''\rx.
 $$
By \eqref{ppp}, the right-hand side equals
$$\sum_{(a)(b)(c)}\lx b',a'\rx T_{\om_{12}(a)}\lx c',a''\rx T_{\om_{23}(a)+\om_{12}(b)+\om_{12}(c)}\lx a^{(4)},b''\rx T_{\om_{34}(a)}\lx a^{(3)},c''\rx. $$
Since $\lx c',a''\rx\in\mg T_{-\om_{12}(c)-\om_{23}(a)}$ and
$\lx a^{(4)},b''\rx\in\mg T_{-\om_{34}(a)-\om_{12}(b)}$, this may be written 
\begin{multline*}\sum_{(a)(b)(c)}\lx b',a'\rx T_{\om_{12}(a)+\om_{12}(b)}\lx a^{(4)},b''\rx T_{\om_{34}(a)}\lx c',a''\rx T_{\om_{23}(a)+\om_{12}(c)}\lx a^{(3)},c''\rx\\
\begin{split}&=\sum_{(a)(b)}\lx b',a'\rx T_{\om_{12}(a)+\om_{12}(b)}\lx a^{(3)},b''\rx T_{\om_{23}(a)}\ep(a'')\ep(c)\\
&=\sum_{(a)(b)}\lx b',a'\rx T_{\om_{12}(a)+\om_{12}(b)}\lx\mu_l(\ep(a'')1) a^{(3)},b''\rx \ep(c)\\
&=\sum_{(a)(b)}\lx b',a'\rx T_{\om_{12}(a)+\om_{12}(b)}\lx a'',b''\rx \ep(c)=
\ep(abc),
\end{split}\end{multline*}
where we used  first \eqref{ud} for the pair $(a,c)$, then \eqref{mrpa}, then
\eqref{cua}, and finally  \eqref{ud}  for the pair $(a,b)$.
\end{proof}

Cobraided  $\gh$-bialgebroids arise naturally from dynamical $R$-matrices, see
  \cite[Remark 3.4]{kn} and, for the case without spectral parameter, \cite[Cor.\ 3.20]{r}. 

\begin{proposition}\label{cbp}
Let
the dynamical $R$-matrix $R$ and the $\gh$-bialgebroid $A$ be related as in \emph{\S \ref{fcs}}. Assume that
the matrix elements
 $R_{ac}^{bd}(\la,z)$ are meromorphic in $\la$ for each  $z\in\Cp$. 
Then, there exists a cobraiding on $A$ defined by
\begin{equation}\label{cbd}\lx L_{ab}(w),L_{cd}(z)\rx=R_{ac}^{bd}(\la,w/z)T_{-\om(a)-\om(c)}. \end{equation}
\end{proposition}

\begin{proposition}\label{up}
If $R$ is unitary in the sense of \eqref{ur}, then the cobraiding described in 
\emph{Proposition \ref{cbp}} is unitary.
\end{proposition}

\begin{proof}
By Lemma \ref{ugl}, it is enough to verify that  \eqref{ud} holds for
 $a\mapsto L_{ab}(z)$ and $b\mapsto L_{cd}(w)$. This amounts to the identity
\begin{equation*}\begin{split}\de_{ab}\de_{cd}T_{-\om(a)-\om(c)}&=\sum_{xy}\lx L_{cx}(w),L_{ay}(z)\rx T_{\om(x)+\om(y)}\lx L_{yb}(z),L_{xd}(w)\rx\\
&=\sum_{xy} R^{xy}_{ca}(\la,z/w) R_{yx}^{bd}(\la,w/z)T_{-\om(a)-\om(c)},
 \end{split}\end{equation*}
which is indeed equivalent to \eqref{ur}. 
\end{proof}

\begin{proposition}\label{acl}
If $A$ is an $\gh$-Hopf algebroid with  a unitary cobraiding, then
$$\lx  S(x),y\rx=T_\al\lx y,x\rx T_\be,\qquad x\in A_{\al\be}. $$
\end{proposition}

\begin{proof}
Expanding $x$ and $y$ using the first expression in \eqref{cua}
gives
$$\lx S(x),y\rx=\sum_{(x)(y)}\lx S(x'')\mu_r(\ep(x')1),\mu_l(\ep(y')1)y''\rx. $$
Using, respectively, \eqref{mrp}, \eqref{ud}, \eqref{pppa}, \eqref{ics} and  \eqref{cbe},
this may be written
\begin{multline*}
\sum_{(x)(y)}T_\al\,\ep(x'y')\,T_{\om_{12}(y)}\,\lx S(x''),y''\rx\\
\begin{split}&=T_\al\sum_{(x)(y)}\lx y',x'\rx\, T_{\om_{12}(x)+\om_{12}(y)}\,\lx x'',y''\rx
\,T_{\om_{23}(y)}\,\lx S(x^{(3)}),y^{(3)}\rx\\
&=T_\al\sum_{(x)(y)}\lx y',x'\rx\, T_{\om_{12}(x)+\om_{12}(y)}\,\lx x''S(x^{(3)}),y''\rx\\
&=T_\al\sum_{(x)(y)}\lx y',x'\rx\, T_{\om_{12}(x)+\om_{12}(y)}\,\lx
\mu_l(\ep(x'')1),y''
\rx\\
&=T_\al\sum_{(x)(y)}\lx y',x'\rx\, T_{\om_{12}(x)+\om_{12}(y)}\,
\ep(y''x'')\,T_\be.
\end{split}\end{multline*}
The same expression is similarly obtained by expanding $x$ and $y$ in 
$T_\al\lx y,x\rx T_\be$ using the second expression in 
\eqref{cua}. 
\end{proof}

The following consequence was taken as an axiom in
 \cite{r}. In the unitary case, we find that it holds automatically.

\begin{corollary} In an $\gh$-Hopf algebroid equipped with  a unitary cobraiding,
\begin{equation}\label{sun}\langle a,b\rangle=S^{D_\gh}\left(\langle S(a),S(b)\rangle\right).\end{equation}
\end{corollary}

\subsection{Algebraic construction of $R$-matrices}
\label{crs}

By Proposition \ref{cbp},
 a dynamical $R$-matrix gives a cobraided  $\gh$-bialgebroid. Conversely,  cobraided  $\gh$-bialgebroids can be used to obtain 
 solutions to the  general QDYB equation \eqref{gdyc}.  We will consider two closely related examples of such constructions. In both cases,  
our starting point is the identity 
\begin{multline}\label{udy}\sum_{(a)(b)(c)}\lx a'',c'\rx T_{\om_{12}(c)}
\lx b'',c''\rx T_{\om_{12}(a)+\om_{12}(b)}\lx a',b'\rx\\
=\sum_{(a)(b)(c)}\lx a'',b''\rx  T_{\om_{12}(a)+\om_{12}(b)}
\lx b',c'\rx T_{\om_{12}(c)}\lx a',c''\rx,
\end{multline}
which is obtained by applying $\lx\cdot,c\rx$ to both sides of \eqref{cba}.

For the first example,  consider a cobraided $\gh$-bialgebroid obtained from a dynamical $R$-matrix as in Proposition \ref{cbp}. 
Introducing  the notation
\begin{equation}\label{arn}\vec L_{ab}(z)=L_{a_1b_1}(z_1)\dotsm L_{a_nb_n}(z_n), \qquad a,b\in X^n, \quad z\in(\mathbb C^\times)^n,\end{equation}
we  write
\begin{equation}\label{pf}
\lx \vec L_{ab}(w),\vec L_{cd}(z)\rx=\delta_{\om(a)+\om(c),\om(b)+\om(d)}\mathcal Z_{ac}^{bd}(\la;w,z)T_{-\om(a)-\om(c)},
\end{equation}
where $\om(a)=\sum_i\om(a_i)$. Substituting
  $a \mapsto\vec L_{da}(u)$, $b\mapsto \vec L_{eb}(w)$ and $c\mapsto\vec L_{fc}(z)$ in \eqref{udy} yields \eqref{gdyc}
in the form
\begin{multline*}\sum_{xyz}\mathcal  Z^{xy}_{de}(\la-\om(f);u,w)
\,\mathcal Z^{az}_{xf}(\la;u,z)\,\mathcal Z^{bc}_{yz}(\la-\om(a);w,z)\\
=\sum_{xyz}\mathcal Z^{yz}_{ef}(\la;w,z) \,\mathcal Z^{xc}_{dz}(\la-\om(y);u,z)\,\mathcal Z^{ab}_{xy}(\la;u,w).
\end{multline*}
Here, $u$, $w$ and $z$ may  be vectors of different dimension.
Moreover, in the unitary case
 we may let $a\mapsto\vec L_{bd}(z)$ and $b\mapsto\vec L_{ac}(w)$ in \eqref{ud}, which gives \eqref{gduc} in the form
$$\sum_{xy}\mathcal Z^{xy}_{ab}(\la;w,z)\,\mathcal Z_{yx}^{dc}(\la;z,w)=\delta_{ac}\delta_{bd}.$$

For the second example,   consider a general cobraided $\gh$-bialgebroid $A$, not necessarily obtained via the FRST construction. For each corepresentation $U$ of $A$, fix homogeneous basis elements $e_x^U$ and  matrix elements $t_{xy}^U$ as in \eqref{med}. We may then write
$$\lx t_{ab}^U,t_{cd}^V\rx=\delta_{\om(a)+\om(c),\om(b)+\om(d)}\mathcal R_{ac}^{bd}(\la;U,V)\,T_{-\om(a)-\om(c)}.$$
Replacing $a\mapsto t_{da}^U$, $b\mapsto t_{eb}^V$, $c\mapsto t_{fc}^W$ in \eqref{udy} gives the QDYB equation \eqref{gdyc}. Moreover,  assuming that the cobraiding is unitary, substituting $a\mapsto t_{bd}^V$, $b\mapsto t_{ac}^U$ in \eqref{ud}  gives \eqref{gduc}.

The quantity $\mathcal R_{ac}^{bd}(\la;U,V)$ is a
matrix element for the action of $t_{cd}^V$ in the representation dual to $U$, see \cite[\S 3.2]{r}. 
It is also a matrix element for the natural intertwiner
$\Phi:\,V\widehat\ot U\ra U\widehat\ot V$.
To be precise \cite[Prop.\ 3.18]{r}, $\Phi$ is given by
\begin{equation}\label{tpi}\Phi(e_c^V\ot e_a^U)=\sum_{bd} \mathcal R_{ac}^{bd}(\la;U;V)e_b^U\ot e_d^V.  \end{equation}

\subsection{Lattice models}
\label{fs}

The quantity $\mathcal Z^{bd}_{ac}$ introduced in \eqref{pf} 
is the partition function for the  lattice model
with Boltzmann weights $R$ and fixed boundary conditions. 
 Though this fact 
 should be expected, it seems not to have been discussed in the literature 
in the present setting,
cobraidings on $\gh$-bialgebroids having been defined only recently \cite{r}.  
Since we will occasionally refer to the lattice model interpretation, we 
 provide a brief explanation here.

Consider the finite  lattice obtained by intersecting $m$ vertical and $n$ horizontal lines. The  intersection points are called vertices. 
Including edges and faces at the boundary, each vertex is surrounded by four edges and four faces. 

As in \S \ref{fcs},  let
 $X$ be a finite set, $\gha$ a complex vector space and $\om:\,X\rightarrow\gha$ an arbitrary function. We  label all edges with elements of $X$ and all faces with elements of $\gha$, so that around each vertex there is a label configuration as in the left part of Figure \ref{labelfig}.

\setlength{\unitlength}{2pt}
\begin{figure}[h]
$\begin{array}{c@{\hspace{1in}}c}
\begin{picture}(40,40)
\put(13,23){$\alpha$}
\put(23,23){$\beta$}
\put(23,13){$\gamma$}
\put(13,13){$\delta$}
\put(19,33){$b$}
\put(33,18){$d$}
\put(19,5){$a$}
\put(4,18){$c$}
\put(20,10){\line(0,1){20}}
\put(10,20){\line(1,0){20}}
\end{picture}
\begin{picture}(40,40)
\put(18,18){$\alpha$}
\put(18,27){$b$}
\put(27,18){$d$}
\put(18,10){$a$}
\put(10,18){$c$}
\put(15,15){\line(0,1){10}}
\put(15,15){\line(1,0){10}}
\put(15,25){\line(1,0){10}}
\put(25,15){\line(0,1){10}}
\end{picture}
\end{array}$
\caption{Local label configurations in vertex and face picture.
Roman characters refer to elements of $X$, Greek characters to elements of $\gha$.
}
\label{labelfig}\end{figure}
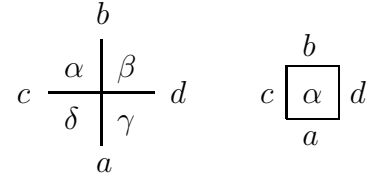

We require that the face labelling is obtained from the edge labelling as follows.  The top left face has label $0$, and as we move east or south, crossing an edge labelled $x$, the face label increases by $\om(x)$.
This is possible if and only if the edge labels around each vertex satisfy
\begin{equation}\label{ir}\om(a)+\om(c)=\om(b)+\om(d). \end{equation}
 If $\om$ is injective, the edge labels are conversely 
determined by the face labels.  A 
 labelling satisfying these rules will be called a \emph{state}.

In our main case of interest,  $\gha=\mathbb C$, $X=\{\pm\}$, $\om(\pm)=\pm 1$.  One may then picture edges labelled $+$ as arrows going up or right and edges labelled 
$-$ as arrows going down or left. The condition \eqref{ir} is  the  ice rule, saying that  each vertex has two incoming and two outgoing edges. In the left part of Figure \ref{fig}, we give an example with $m=3$ and $n=2$.

\begin{figure}[h]
\begin{center}
$\begin{array}{c@{\hspace{1in}}c}
 \begin{pspicture}(0.9,0.9)(2.8,2.1)
\psset{xunit=.7,yunit=.7}
\psset{arrowsize=5pt}
\rput(0.5,0.5){$0$}
\rput(0.4,1.5){$-1$}
\rput(0.5,2.5){$0$}
\rput(1.5,0.5){$1$}
\rput(1.5,1.5){$0$}
\rput(1.5,2.5){$1$}
\rput(2.5,0.5){$0$}
\rput(2.5,1.5){$1$}
\rput(2.5,2.5){$2$}
\rput(3.5,0.5){$1$}
\rput(3.5,1.5){$2$}
\rput(3.5,2.5){$1$}
\rput(0,1){\iceR}
\rput(1,2){\iceL}
\rput(1,0){\iceU}
\rput(1,1){\iceU}
\rput(1,2){\iceU}
\rput(1,1){\iceR}
\rput(2,2){\iceL}
\rput(2,1){\iceD}
\rput(2,1){\iceU}
\rput(2,2){\iceU}
\rput(3,1){\iceL}
\rput(3,2){\iceL}
\rput(3,0){\iceU}
\rput(3,1){\iceU}
\rput(3,3){\iceD}
\rput(4,1){\iceL}
\rput(3,2){\iceR}
\end{pspicture}
&
\begin{pspicture}(0.5,0.5)(2.1,1.4)
\psset{xunit=.7,yunit=.7}
\psset{arrowsize=5pt}
\rput(0.5,0.5){$-1$}
\rput(0.5,1.5){$0$}
\rput(1.5,0.5){$0$}
\rput(1.5,1.5){$1$}
\rput(2.5,0.5){$1$}
\rput(2.5,1.5){$2$}
\rput(0,1){\iceD}
\rput(0,1){\iceU}
\rput(0,0){\iceR}
\rput(0,1){\iceR}
\rput(0,2){\iceR}
\rput(1,1){\iceD}
\rput(1,1){\iceU}
\rput(2,0){\iceL}
\rput(1,1){\iceR}
\rput(1,2){\iceR}
\rput(2,0){\iceU}
\rput(2,1){\iceU}
\rput(2,0){\iceR}
\rput(2,1){\iceR}
\rput(3,2){\iceL}
\rput(3,0){\iceU}
\rput(3,2){\iceD}
\end{pspicture}
\end{array}$
\end{center}
\vspace{7mm}
\caption{Vertex and face picture of a state.}
\label{fig}
\end{figure}
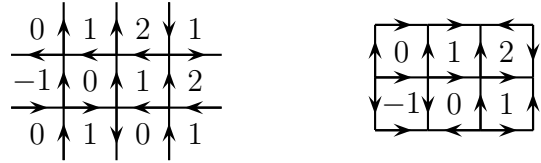

To obtain a statistical model we must  specify a 
 weight function on the  states. To this end,
we fix   parameters $w_1,\dots,w_m,z_1,\dots,z_n\in\mathbb C^\times$ and $\la\in\gha$. 
Consider the vertex at the intersection of the $i$:th 
horizontal line from the top with the  $j$:th
vertical line from the left. If the ambient labelling is as in Figure \ref{labelfig}, then this vertex is assigned weight 
$R^{bd}_{ac}(\la-\al,w_j/z_i). $
The weight of a state is then the product of the weights of all vertices. As an example, the state in Figure 1 has weight
\begin{multline*}R^{+-}_{+-}(\la,w_1/z_1)R^{+-}_{+-}(\la-1,w_2/z_1)R^{-+}_{+-}(\la-2,w_3/z_1)\\
\times R^{++}_{++}(\la+1,w_1/z_2)R^{+-}_{-+}(\la,w_2/z_2)R^{+-}_{+-}(\la-1,w_3/z_2).
\end{multline*}

We have described the  model as a vertex model. 
In the physics literature, an alternative description 
as a face model is  more common.   This corresponds to passing to the dual lattice, interchanging the roles of vertices and faces. In this process, horizontal edges become vertical and vice versa. The face labels  end up at vertices, but we move each one to its neighbouring face in the south-east direction. Labels along the east and south boundary are lost, but these 
do not enter in the partition functions, and can in any case 
 be recovered from the edge labelling. 
In this dual picture, the Boltzmann weights  correspond to interaction round a face rather than a vertex, see  the right part of Figures \ref{labelfig} and \ref{fig} (in Figure \ref{fig}, arrows have been rotated $90^\circ$ clockwise).

We will now consider the model with arbitrary fixed boundary conditions. Given $a,b\in X^m$, $c,d\in X^n$, suppose  the boundary edges are labelled as in Figure~\ref{bfig}. We may then introduce the partition function
\begin{equation}\label{ppf}\mathcal Z^{bd}_{ac}(\la;w;z)=\sum_{\substack{\text{states with}\\ \text{fixed boundary}}}\operatorname{weight}(\text{state}).\end{equation}
We claim that this definition agrees with
\eqref{pf}. To see why, we use 
 \eqref{pppa} on a decomposition
$$\lx\vec L_{ab}(w),\vec L_{cd}(z)\rx=\lx\vec L_{a'b'}(w')\vec L_{a''b''}(w''),\vec L_{cd}(z)\rx $$
(where $w=(w',w'')$ and so on), obtaining the relation
$$\mathcal Z_{ac}^{bd}(\la;w;z)=\sum_{x\in X^n}\mathcal Z_{a'c}^{b'x}(\la;w';z)\mathcal Z_{a''x}^{b''d}(\la-\om(b');w'';z) $$
where $\mathcal Z$ is as in \eqref{pf}.
Similarly, \eqref{pppb} gives
$$\mathcal Z_{ac}^{bd}(\la;w;z)=\sum_{x\in X^m}\mathcal Z_{xc'}^{bd'}(\la;w;z')\mathcal Z_{ac''}^{xd''}(\la-\om(c');w;z''). $$
It is clear that the partition function 
defined in \eqref{ppf}
satisfies the same relations, corresponding to a vertical or horizontal splitting of the lattice. Thus, the equivalence of the two definitions follows by induction on $m$ and $n$.

\begin{figure}[h]
\begin{picture}(80,60)
\put(18,5){$a_1$}
\put(48,5){$a_m$}
\put(18,53){$b_1$}
\put(48,53){$b_m$}
\put(3,38){$c_1$}
\put(3,18){$c_n$}
\put(63,38){$d_1$}
\put(63,18){$d_n$}
\put(33,5){$\dotsm$}
\put(33,45){$\dotsm$}
\put(13,28){$\vdots$}
\put(55,28){$\vdots$}
\put(10,40){\line(1,0){50}}
\put(10,20){\line(1,0){50}}
\put(20,10){\line(0,1){40}}
\put(50,10){\line(0,1){40}}
\end{picture}
\caption{Fixed boundary conditions.}
\label{bfig}\end{figure}
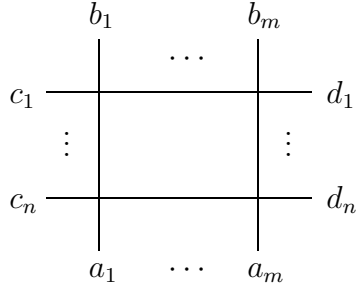

\section{Felder's elliptic quantum group}
\label{fgs}

\subsection{Felder's  quantum group as an $\gh$-Hopf algebroid}
\label{ess}

We will study a particular $\gh$-Hopf algebroid, with $\gha=\mathbb C$, constructed from the  $R$-matrix of the 8VSOS model \cite{b}.
The corresponding  quantum group  was introduced by Felder \cite{f} on the level of its representations, and further studied by Felder and Varchenko \cite{fv}, before it was included in the $\gh$-Hopf algebroid framework of
 Etingof and Varchenko \cite{ev}.

Let $p$ and $q$ be fixed parameters, with $|p|<1$ and $q\neq 0$. 
We  fix a choice of $\log q$ and write
$q^\la=e^{\la\log q}$  for $\la\in\mathbb C$.  
 Compared to the conventions of many authors, we  replace $q^2$ by $q$.

We will use the notation
$$\tha(x)=\prod_{j=0}^\infty(1-p^jx)(1-p^{j+1}/x), $$
$$\tha(x_1,\dots,x_n)=\tha(x_1)\dotsm\tha(x_n), $$
$$(x)_k=\theta(x)\theta(qx)\dotsm\theta(q^{k-1}x), $$
$$(x_1,\dots,x_n)_k=(x_1)_k\dotsm(x_n)_k, $$
for theta functions and elliptic Pochhammer symbols.
We will  freely use elementary identities such as
$$(x)_k=(-1)^kq^{\binom k2}x^k (q^{1-k}/x)_k, $$
see \cite[\S 11.2]{gr}.

Let $\gha=\C$,  $X=\{\pm \}$, and  $\omega(\pm )=\pm 1$. 
We will often identify $\pm 1=\pm$.
Then, 
$$R(\la,z)=\left(\begin{matrix} R_{++}^{++} & 0 & 0 & 0\\
0& R_{+-}^{+-} & R^{-+}_{+-} & 0\\
0 & R^{+-}_{-+} & R_{-+}^{-+} & 0\\
0 & 0 & 0 & R^{--}_{--} 
\end{matrix}\right)= \left(\begin{matrix} 1 & 0 & 0 & 0\\
0& a(\la,z) & b(\la,z) & 0\\
0 & c(\la,z) & d(\la,z) & 0\\
0 & 0 & 0 & 1 
\end{matrix}\right), $$
where
$$a(\la,z)=\frac{\tha(z,q^{\la+2})}{\tha(qz,q^{\la+1})},\qquad 
b(\la,z)=\frac{\tha(q,q^{-\la-1}z)}{\tha(qz,q^{-\la-1})},
 $$
$$c(\la,z)=\frac{\tha(q,q^{\la+1}z)}{\tha(qz,q^{\la+1})},\qquad 
d(\la,z)=\frac{\tha(z,q^{-\la})}{\tha(qz,q^{-\la-1})},
 $$
satisfies the QDYB equation \eqref{dyx} and the unitarity relation \eqref{ur}.

Applying the FRST construction of \S \ref{fcs},
 one obtains an $\gh$-bialgebroid $A$.
The generators will be denoted
$$\alpha(z)=L_{++}(z),\quad \be(z)=L_{+-}(z),\quad \ga(z)=L_{-+}(z),\quad \de(z)=L_{--}(z). $$
In analogy with \eqref{arn}, we use notation such as
 $$\vec \alpha(z)=\alpha(z_1)\dotsm \alpha(z_n),\qquad z\in (\mathbb C^\times)^n.$$
We refer to \cite{knr} for an explicit list of relations,
 noting  only that
\begin{equation}\label{lc}L_{ab}(z)L_{ab}(w)=L_{ab}(w)L_{ab}(z),\qquad a,b\in\{\pm\}, \end{equation}
and that
\begin{subequations}\label{gar}
\begin{align}\label{gara}
\ga(w)\al(z)&=a(\la,z/w)\al(z)\ga(w)+b(\la,z/w)\ga(z)\al(w),\\
\label{garb}\al(w)\ga(z)&=c(\la,z/w)\al(z)\ga(w)+d(\la,z/w)\ga(z)\al(w).
\end{align}
\end{subequations}

The $\gh$-bialgebroid $A$ can be extended to an $\gh$-Hopf algebroid, which we  denote $\E$.
To define it,  let
$$\det(z)=\frac{F(\mu)}{F(\la)}\big(\al(z)\de(qz)-\ga(z)\be(qz)\big), $$
where
\begin{equation}\label{f}F(\la)=q^{-\frac\la 2}\tha(q^{\la+1}). \end{equation}
Then, $\E$ is obtained from $A$ by
adjoining the inverses $\det^{-1}(z)$, which are required to be central elements with $\De(\det^{-1}(z))= \det^{-1}(z)\ot\det^{-1}(z)$, $\ep(\det^{-1}(z))=1$, 
and defining the antipode by $S({\det}^{-1}(z))=\det(z)$,
\begin{align*}
 S(\al(z))&= \frac{F(\mu)}{F(\la)}\,{\det}^{-1}(q^{-1}z)\de(q^{-1}z),\\
 S(\be(z))&=- \frac{F(\mu)}{F(\la)}\,{\det}^{-1}(q^{-1}z)\be(q^{-1}z),\\
 S(\ga(z))&=-\frac{F(\mu)}{F(\la)}\, {\det}^{-1}(q^{-1}z)\ga(q^{-1}z),\\
 S(\de(z))&=\frac{F(\mu)}{F(\la)}\,{\det}^{-1}(q^{-1}z)\al(q^{-1}z).\end{align*}

\subsection{Singular cobraiding}

Since $R(\la,z)$ is singular at $z=q^{-1}$, 
 Proposition \ref{up} cannot be applied, so $A$ does not strictly speaking have a unitary cobraiding. An apparent solution to this problem is to
use instead of \eqref{cbd} the definition
$$\lx L_{ab}(w),L_{cd}(z)\rx_{\operatorname{reg}}=\tha(qw/z)R_{ac}^{bd}(\la,w/z)T_{-\om(a)-\om(c)}.$$
 This gives a \emph{bona fide} cobraiding  $\lx\cdot,\cdot\rx_{\operatorname{reg}}$ on $A$, which satisfies a modified
unitarity axiom, see \cite{h}. However, trying to extend this cobraiding 
to the  $\gh$-Hopf algebroid $\E$, singularities reappear. 
We prefer to stick to the definition \eqref{cbd}, which leads to a  
\emph{singular} cobraiding, defined on a subspace of 
 $A\times A $. 
 Explicitly, the singular cobraiding is defined on the generators by
\begin{multline}\label{cb}
\left(\begin{matrix}\lx\al(w),\al(z)\rx &\lx\al(w),\be(z)\rx&\lx\al(w),\ga(z)\rx&\lx\al(w),\de(z)\rx\\
\lx\be(w),\al(z)\rx &\lx\be(w),\be(z)\rx&\lx\be(w),\ga(z)\rx&\lx\be(w),\de(z)\rx\\
\lx\ga(w),\al(z)\rx &\lx\ga(w),\be(z)\rx&\lx\ga(w),\ga(z)\rx&\lx\ga(w),\de(z)\rx\\
\lx\de(w),\al(z)\rx &\lx\de(w),\be(z)\rx&\lx\de(w),\ga(z)\rx&\lx\de(w),\de(z)\rx\end{matrix}\right)\\
=\left(\begin{matrix}T_{-2}&0&0&a(\la,w/z)T_0\\0&0&b(\la,w/z)T_0&0\\0&c(\la,w/z)T_0&0&0\\d(\la,w/z)T_0&0&0&T_2\end{matrix}\right).
\end{multline}
The cobraiding axioms then give a meaning to  expressions of the form
$$\lx f(\la,\mu)L_{a_1b_1}(w_1)\dotsm L_{a_mb_m}(w_m),
g(\la,\mu) L_{c_1d_1}(z_1)\dotsm L_{c_nd_n}(z_n)\rx,
 $$
where $w_i/z_j\notin p^\mathbb Z q^{-1}$ for all $i,\,j$.
As long as we restrict to the subspace spanned by such expressions, the properties of unitary cobraidings discussed in \S \ref{cbss} remain  valid. 

In \cite{kn}, the singular cobraiding was extended to $\E$, by defining
$$\lx L_{ab}(w),{\det}^{-1}(z)\rx=\delta_{ab}q^{-\frac 12}\frac{\theta(w/z)}{\theta(w/qz)}T_{-a}, $$
$$\lx {\det}^{-1}(w), L_{ab}(z)\rx=\delta_{ab}\,q^{-\frac 12}\frac{\theta(qw/z)}{\theta(w/z)}T_{-a}, $$
$$\lx{\det}^{-1}(w),{\det}^{-1}(z)\rx=q\frac{\theta(w/qz)}{\theta(qw/z)}. $$
It is easy to see that, for generic spectral
parameters, the cobraiding and unitarity axioms remain valid. 
The following result is then easily proved by induction.

\begin{lemma}\label{dcgl}
For generic $(x,w,y,z)\in(\Cp)^k\times (\Cp)^l\times (\Cp)^m\times (\Cp)^n$,
\begin{multline*}\lx\overrightarrow{\det}^{-1}(x)\vec L_{ab}(w),\overrightarrow{\det}^{-1}(y)\vec L_{cd}(z)\rx=q^{km-\frac12(kn+lm)}\prod_{1\leq i\leq k,\,1\leq j\leq m}\frac{\theta(x_i/qy_j)}{\theta(qx_i/y_j)}\\
\times
\prod_{1\leq i\leq k,\,1\leq j\leq n}\frac{\theta(qx_i/z_j)}{\theta(x_i/z_j)}
\prod_{1\leq i\leq l,\,1\leq j\leq m}\frac{\theta(w_i/y_j)}{\theta(w_i/qy_j)}\lx\vec L_{ab}(x),\vec L_{cd}(z)\rx. \end{multline*}
\end{lemma}

The existence of the antipode is related to the following  symmetry, which  in physical terms is the crossing symmetry for the partition function.
Here and below, we write
$$|x|=x_1+\dots+x_n,\qquad x\in\{\pm\}^n= \{\pm 1\}^n. $$

\begin{lemma}\label{pfs}
For $(w,z)\in(\mathbb C^\times)^m\times(\mathbb C^\times)^n$ generic, $a,b\in\{\pm\}^m$ and $c,d\in\{\pm\}^n$,
\begin{multline*}\lx\vec L_{ab}(w),\vec L_{cd}(z)\rx
=(-1)^{\frac 12(|c|-|d|)}q^{-\frac {1}2\,mn}\prod_{1\leq i\leq n,\,1\leq j\leq m}
\frac{\tha(z_i/w_j)}{\tha(z_i/qw_j)}\\
\times\prod_{j=1}^n\frac{F(\la-c_1-\dots-c_j)}
{F(\la-d_1-\dots-d_j-|b|)}\, T_{-|c|}\lx\vec  L_{-d^{\op},-c^{\op}}(q^{-1}z),\vec L_{ab}(w)\rx T_{-|d|},
 \end{multline*}
with $F$ as in \eqref{f} and where we write $(x_1,\dots,x_n)^{\op}=(x_n,\dots,x_1)$.
\end{lemma}

\begin{proof}
Let $x=\vec L_{cd}(z)$ and $y=\vec L_{ab}(w)$ in Proposition \ref{acl}. It is straight-forward to check that 
\begin{multline}\label{sv}S(\vec L_{ab}(z))=(-1)^{\frac12(|a|-|b|)}\prod_{j=1}^n\frac{F(\mu+a_{j+1}+\dots+a_n)}{F(\la+b_{j+1}+\dots+b_n)}\\
\times\overrightarrow{{\det}}^{-1}(q^{-1}z)\vec L_{-b^{\op},-a^{\op}}(q^{-1}z).
\end{multline}
Using also  Lemma \ref{dcgl} one arrives at the desired result.
\end{proof}

Finally, we mention the following useful algebra symmetries. 

\begin{proposition}\label{ip}
There exists an algebra antiautomorphism $\phi$ of $\E$ defined on the generators by $\phi(f(\la))=f(-\la-2)$, $\phi(f(\mu))=f(-\mu-2)$,
$\phi(L_{ab}(z))=L_{ab}(z^{-1})$, $\phi(\det^{-1}(z))=\det^{-1}(q^{-1}z^{-1})$. It satisfies $\phi\circ S=S^{-1}\circ\phi$,
\begin{equation}\label{dph}(\phi\ot\phi)\circ\De=\De\circ\phi,\qquad \phi^{\D}\circ\ep\circ\phi=\ep, \end{equation}
\begin{equation}\label{phun}\lx x,y\rx=\phi^{\D}(\lx\phi(y),\phi(x)\rx), \end{equation}
where $\phi^{\D}$ is the algebra antiautomorphism of $\D$ defined by $\phi^{\D}(f(\la))=f(-\la-2)$, $\phi^{\D}(T_\al)=T_\al$. Moreover, there exists an algebra automorphism $\psi$ of   $\E$ defined by  $\psi(f(\la))=f(-\la-2)$, $\psi(f(\mu))=f(-\mu-2)$,
$\psi(L_{ab}(z))=L_{-a,-b}(z)$, $\psi(\det^{-1}(z))=\det^{-1}(z)$. It satisfies
$S\circ\psi=\psi\circ S$,
$$(\psi\ot\psi)\circ\De=\De\circ\psi,\qquad \psi^{\D}\circ\ep\circ\psi=\ep, $$
\begin{equation}\label{psu}\lx x,y\rx=\psi^{\D}(\lx\psi(x),\psi(y)\rx), \end{equation}
where $\psi^{\D}$ is the algebra automorphism of $\D$ defined by $\psi^{\D}(f(\la))=f(-\la-2)$, $\psi^{\D}(T_\al)=T_{-\al}$.
\end{proposition}

 Proposition \ref{ip} can be proved in a straight-forward manner. The most tedious part is to verify that, in order to prove \eqref{phun} and \eqref{psu}, it is enough to check them for a set of generators. In the Appendix, this is done in a systematic way. 

Note that
$\phi\circ\phi=\psi\circ\psi=\id$ and $\psi\circ\phi=\phi\circ\psi$.
If we write $x^\ast=\psi(\phi(x))=\phi(\psi(x))$, then  
\begin{equation}\label{ic}\lx x,y\rx=S^{\D}(\lx y^\ast,x^\ast\rx), 
\end{equation}
where $S^{\D}=\psi^{\D}\circ\phi^{\D}$ is as in \eqref{sdd}. 
 The  $\mathbb C$-linear involution $\ast$ is a slight  modification of the $\mathbb C$-antilinear involution used in \cite{knr}.

\subsection{Elementary properties of the cobraiding}\label{scs}

We conclude \S \ref{ess} with some useful identities involving the singular cobraiding on  Felder's quantum group $\mathcal E$.

\begin{lemma}\label{adl}
For generic $(w,z)\in (\mathbb C^\times)^m\times(\mathbb C^\times)^n $,
\begin{subequations}
\begin{align}
\label{aca}\lx \vec \alpha(w),\vec \alpha(z)\rx&=T_{-n-m},\\
\label{acd}\lx \vec \alpha(w),\vec \delta(z)\rx
&=\frac{(q^{\la+2+n-m})_m}{(q^{\la+2-m})_m}\prod_{1\leq i\leq m,\,1\leq j\leq n}\frac{\theta(w_i/z_j)}{\theta(qw_i/z_j)}\, T_{n-m}.
\end{align}
\end{subequations}
\end{lemma}

\begin{proof}
To prove \eqref{aca}, we use \eqref{pppa} to write
$$\lx \vec \alpha(w),\vec \al(z)\rx=\sum_{b\in\{\pm\}^n} \lx\al(w_1),\vec L_{+b}(z)\rx T_{|b| } \lx \vec \alpha(w_2,\dots,w_m),\vec L_{b+}(z)\rx.$$
By \eqref{cbv}, only the term with $b_1=\dots=b_n=+$ is non-zero, so
$$\lx \vec \alpha(w),\vec \al(z)\rx= \lx\al(w_1),\vec \al(z)\rx T_{n} \lx \vec \alpha(w_2,\dots,w_m),
\vec \alpha(z)\rx.$$
By induction on $m$, this reduces the proof of \eqref{aca} to the case  $m=1$.
In that special case, we similarly use 
 \eqref{pppb} to write
$$\lx \al(w),\vec \al(z)\rx
=\lx\alpha(w),\al(z_1)\rx T_{1}\dotsm T_{1}\lx\alpha(w),\al(z_n)\rx
=T_{-n-1}
.$$
The  identity \eqref{acd} now  follows using Lemma \ref{pfs}.
\end{proof}

\begin{lemma}\label{pol}
For  $a\in\{\pm\}^n$,
 $x\in\E$ and generic $(w,z)\in (\mathbb C^\times)^m\times(\mathbb C^\times)^n $,
\begin{subequations}
\begin{align}
\label{poa}\lx \vec \alpha(w) x,\vec L_{-a}(z)\rx
&=\prod_{1\leq i\leq m,\,1\leq j\leq n}\frac{\theta(w_i/z_j)}{\theta(qw_i/z_j)}\frac{(q^{\la+2+n-m})_m}{(q^{\la+2-m})_m}\, T_{-m} \lx x,\vec L_{-a}(z)\rx, \\
\label{pob}\lx \vec L_{+a}(z),x\vec \alpha(w)\rx&=\lx \vec L_{+a}(z),x\rx T_{-m},\\
\label{pod}\lx \vec L_{+a}(z),x\vec \delta(w)\rx&=\prod_{1\leq i\leq m,\,1\leq j\leq n}\frac{\theta(z_j/w_i)}{\theta(qz_j/w_i)}
\lx \vec L_{+a}(z),x\rx\frac{(q^{\la+2+m})_n}{(q^{\la+2})_n}\,T_{m}.
\end{align}
\end{subequations}
\end{lemma}

\begin{proof}
Similarly as in the proof of  \eqref{aca}, 
\begin{equation*}\begin{split}\lx\vec \alpha(w)  x,\vec L_{-a}(z)\rx&=\sum_{b\in\{\pm\}^n} \lx \vec \alpha(w) ,\vec L_{-b}(z)\rx T_{|b|} \lx x,\vec L_{ba}(z)\rx\\
&= \lx\vec \alpha(w),\vec \delta(z)\rx T_{-n} \lx x,\vec L_{-a}(z)\rx.
\end{split}\end{equation*} 
Thus, \eqref{poa} follows  from \eqref{acd}. The other statements are proved similarly.
\end{proof}

\begin{proposition}\label{lll}
For $a,b,c,d\in\{\pm\}^n$ and generic $z\in(\mathbb C^\times)^n$,
$$\lx \vec L_{ab}(z),\vec L_{cd}(z)\rx
=\delta_{ad}\delta_{bc}\,T_{-|a|-|c|}.$$
\end{proposition}

\begin{proof}
We proceed by induction over $n$. The case $n=1$  follows  from 
\eqref{cb}.
Writing $\hat a=(a_1,\dots,a_{n-1})$ for $a\in\mathbb C^n$, 
\eqref{ppp} gives 
\begin{multline*}
\lx \vec L_{ab}(z),\vec L_{cd}(z)\rx
=\sum_{x\in\{\pm \}^n}
\lx \vec L_{\hat a\hat b}(\hat z),\vec L_{cx}(z)\rx
 T_{|x|}
\lx  L_{a_nb_n}(z_n),\vec L_{xd}(z)\rx\\
=\sum_{x,y\in\{\pm \}^n}
\lx\vec  L_{\hat y\hat b}(\hat z),\vec  L_{\hat c\hat x}(\hat z)\rx T_{|\hat y|}
\lx\vec  L_{\hat a\hat y}(\hat z), L_{c_nx_n}(z_n)\rx\\
\times T_{|x|} \lx  L_{y_nb_n}(z_n),\vec L_{\hat x\hat d}(\hat z)\rx
T_{y_n}\lx  L_{a_ny_n}(z_n), L_{x_nd_n}(z_n)\rx.
\end{multline*}
By the induction hypothesis, all terms vanish except those with $y= x$, and the expression simplifies to 
\begin{multline*}
\delta_{\hat b\hat c} \delta_{a_nd_n} T_{-|\hat c|}
\sum_{x\in\{\pm \}^n}\lx\vec  L_{\hat a\hat x}(\hat z), L_{c_nx_n}(z_n)\rx T_{|x|} \lx  L_{x_nb_n}(z_n),\vec L_{\hat x\hat d}(\hat z)\rx
T_{-a_n}\\
=\delta_{\hat b\hat c}\delta_{a_nd_n}
 T_{-|\hat c|}
\sum_{(u)(v)}
\lx v', u'\rx
T_{\om_{12}(u)+\om_{12}(v)} \lx u'',v''\rx
T_{-a_n},
\end{multline*}
where $u=L_{c_nb_n}(z_n)$, $v=\vec L_{\hat a\hat d}(\hat z)$.
Applying the unitarity axiom \eqref{ud} completes the proof.
\end{proof}

The following corollary will be extremely useful.

\begin{corollary}\label{rl}
For $a,c\in\{\pm\}^m$,
 $b,d\in\{\pm\}^n$, and $(w,z)\in(\mathbb C^\times)^m\times(\mathbb C^\times)^n$
generic,
\begin{equation}\label{li}\lx \vec L_{+a}(w)\vec L_{+b}(z),\vec L_{d+}(z) \vec L_{c+}(w)\rx
=\lx  \vec L_{ca}(w),\vec L_{db}(z)\rx T_{-m-n}.\end{equation}
\end{corollary}

\begin{proof}
Straight-forward expansion using \eqref{ppp} gives
\begin{multline*}\lx \vec L_{+a}(w)\vec L_{+b}(z),\vec L_{d+}(z) \vec L_{c+}(w)\rx
\\
= \sum_{e,f\in\{\pm\}^m,\,g,h\in\{\pm\}^n}\lx \vec L_{fa}(w),\vec L_{dg}(z)\rx
T_{|f|}\lx\vec  L_{+f}(w),\vec L_{ce}(w)\rx\\
\times T_{|e|+|g|}\lx\vec L_{hb}(z),\vec L_{g+}(z)\rx T_{|h|}\lx\vec L_{+h}(z),\vec L_{e+}(w)\rx.
\end{multline*}
By Proposition \ref{lll}, the only non-vanishing term in the sum is
 \begin{multline*} \lx\vec L_{ca}(w),\vec L_{db}(z)\rx
T_{|c|}\lx\vec  L_{+c}(w),\vec L_{c+}(w)\rx T_{m+n}\lx\vec L_{+b}(z),\vec L_{b+}(z)\rx T_{n}\lx\vec L_{++}(z),\vec L_{++}(w)\rx\\
=\lx  \vec L_{ca}(w),\vec L_{db}(z)\rx T_{-m-n},\end{multline*}
where we also used \eqref{aca}.
\end{proof}

In terms of the 8VSOS model, 
the right-hand side of \eqref{li} gives
 the partition function 
on a rectangular lattice, with arbitrary fixed boundary conditions, while the left-hand side gives the partition function on a square lattice, with the east and south boundary fixed as domain walls (see \S \ref{ews}).

\section{Embedded corepresentations}\label{ecs}

Fixing a non-negative integer $N$, 
 we identify subsets 
 $S\subseteq [N]=\{1,\dots,N\}$ with elements of $\{\pm\}^N$ through
$$S_i=\begin{cases}+,& i\in S,\\
-, & i\notin S.
\end{cases} $$
For $S\subseteq [N]$ and $z\in(\mathbb C^\times)^N $, we introduce the elements
$$e_S(z)=\vec\gamma(z_{S^c})\vec\alpha(z_S), \qquad
E_S(z)=\vec L_{S+}(z). $$
Here, we use notation such as $\vec\alpha(z_S)=\prod_{i\in S}\al(z_i)$, which is well-defined in view of \eqref{lc}.

\begin{lemma}\label{vl}
For generic $z\in(\mathbb C^\times)^N$,  
\begin{equation}\label{vn}\operatorname{span}_{f\in\mg,\,S\subseteq [N]}
\{\mu_l(f)e_S(z)\}
=\operatorname{span}_{f\in\mg,\,S\subseteq [N]}
\{\mu_l(f)E_S(z)\}.
\end{equation}
\end{lemma}

\begin{proof} 
Iterating the commutation relations \eqref{gar} will expand $e_S$ 
as a sum of the $E_T$, with coefficients in $\mu_l(\mg)$. Though \eqref{gar} is not applicable when $qz/w\in p^\mathbb Z$, that obstruction does not arise  for generic $z$. 
Conversely, iterating \eqref{gara} in the form
$$\al(z)\ga(w)=\frac 1{a(\la,z/w)}\,\ga(w)\al(z)-\frac{b(\la,z/w)}{a(\la,z/w)}\,\ga(z)\al(w) $$
will  expand  $E_S$ 
as a sum of the $e_T$, as long as $z$ is generic. 
\end{proof}

\begin{remark}\label{tr}
If $S=\{s_1<\dots<s_m\}$ and  $T=\{t_1<\dots<t_m\}$, write $S\leq T$ if $s_i\leq t_i$ for all $i$.
 It is then easy to check that, for generic $z$, $E_S\in\operatorname{span}_{f\in\mg,\,T\geq S}
\{\mu_l(f)e_T(z)\}$,  $e_S\in\operatorname{span}_{f\in\mg,\,T\geq S}
\{\mu_l(f)E_T(z)\}$. 
\end{remark}

We  denote by $V(z)$ the space \eqref{vn}, viewed as an $\gh$-space with scalar multiplication $fv=\mu_l(f)v$ and grading corresponding to the left grading in $\E$, that is, $e_S(z),\, E_S(z)\in V_{2|S|-N}(z)$. Since
\begin{equation}\label{bme}\Delta(E_S(z))=\sum_{T\subseteq [N]} \vec L_{ST}(z)\ot E_T(z), \end{equation}
$\Delta|_{V(z)}$ is a corepresentation of  $\E$.

Next, we introduce the dual elements
$$f_S(z)=\vec\alpha(z_S)\vec\beta(z_{S^c}),\qquad F_S(z)=\vec L_{+S}(z). $$
Similarly as in Lemma \ref{vl}, one can show that, for generic $z$,
$$\operatorname{span}_{f\in\mg,\,S\subseteq [N]}
\{\mu_r(f)f_S(z)\}
=\operatorname{span}_{f\in\mg,\,S\subseteq [N]}
\{\mu_r(f)F_S(z)\}.$$
This space will be denoted $W(z)$.

\begin{proposition} \label{dbl}
For generic $z\in(\mathbb C^\times)^N$, 
$$\lx F_T(z),E_S(z)\rx=\delta_{ST}T_{-2m},$$
\begin{equation}\label{feo}\lx f_T(z),e_S(z)\rx=\delta_{ST}A_{S,z} T_{-2m},
 \end{equation}
where $m=|S|$ and where
\begin{equation}\label{ph}A_{S,z}(\la)=
\frac{(q^{\la+2+N-2m})_{m}}{(q^{\la+2-m})_{m}}
\prod_{i\in S,j\in S^c}\frac{\tha(z_i/z_j)}{\tha(qz_i/z_j)}.\end{equation}
\end{proposition}

\begin{proof}
The first identity is a special case of Proposition \ref{lll}.

As for \eqref{feo}, we first note that
\eqref{cbv} implies that $\lx f_T(z),e_S(z)\rx$ vanishes unless $|S|=|  T|$. Using first \eqref{pob} and then \eqref{poa}, we may 
 pull out all factors involving
 $\alpha$. This leads to an expression containing the factor 
$\prod_{i\in T,j\in S^c}\theta(z_i/z_j)$,
which vanishes unless $T\subseteq S$. Thus, we may assume $T=S$, in which case
we obtain
$$\lx f_S(z),e_S(z)\rx=A_{S,z}
T_{-m}\lx\vec\be(z_{S^c}),\vec\ga(z_{S^c}) \rx T_{-m},
 $$
where,
by  Proposition~\ref{lll}, $\lx\vec\be(z_{S^c}),\vec\ga(z_{S^c}) \rx=1$.
\end{proof}

\begin{corollary}\label{vec}
For generic $z\in(\mathbb C^\times)^N$, $(e_S(z))_{S\subseteq [N]}$ and   $(E_S(z))_{S\subseteq [N]}$ are bases for the space  $V(z)$ over $\mg$. In particular,
$\dim_{\mg} V(z)=2^N$. Moreover, any $x\in V(z)$ can be written 
$$x=\sum_{S\subseteq [N]}\mu_l(\lx F_S,x\rx 1)  E_S(z)
=\sum_{S\subseteq [N]}\mu_l(A_{S,z}^{-1}\lx f_S,x\rx 1) e_S(z).
  $$
Similarly, any $x\in W(z)$ can be written
$$x=\sum_{S\subseteq [N]}\mu_r(\lx x,E_S\rx 1)  F_S(z)
=\sum_{S\subseteq [N]}\mu_r(A_{S,z}^{-1}\lx x,e_S\rx 1) f_S(z).$$
\end{corollary}

\begin{proof} Any $x\in V(z)$ can be written $x=\sum_S C_S(\la)E_S(z)$, for some $C_S\in\mg$.   Proposition \ref{dbl} then gives $C_S=\lx F_s,x\rx 1$. In particular, the expansion is unique, so $(E_S(z))_{S\subseteq [N]}$ form a basis. Similar arguments apply for the other cases. 
\end{proof}

Since $E_S(z)$ form a basis of  $V(z)$,  \eqref{bme} exhibits $\vec L_{ST}(z)$ as a matrix element of that corepresentation. 
By \eqref{tpme}, the fact that the  matrix elements factor means that
\begin{equation}\label{tpd}V(z)\simeq V(z_1)\wh\ot\cdots\wh\ot V(z_N)
\end{equation}
as corepresentations.

We conclude with two results that will be needed later.

\begin{lemma}\label{pool}
For  $(w,z)\in (\mathbb C^\times)^n\times(\mathbb C^\times)^N$
generic,
$u\in W(z)$, $v\in V(z)$ and $a,b,c,d\in\{\pm\}^n$,
$$\lx\vec L_{ab}(w)u,\vec L_{cd}(w)v\rx=\delta_{ad}\delta_{bc}T_{-|b|}\lx u,v\rx T_{-|a|}, $$
$$\lx u\vec L_{ab}(w),v \vec L_{cd}(w)\rx=\delta_{ad}\delta_{bc}\lx u,v\rx T_{-|a|-|b|}. $$
\end{lemma}

\begin{proof}
We may choose  $u=f(\mu)\vec L_{+x}(z)$ and $v=g(\la)\vec L_{y+}(z)$. The result then follows from Proposition \ref{lll}.
\end{proof}

\begin{lemma}\label{dac}
For generic $z\in (\mathbb C^\times)^N$,
\begin{align}\label{da}\Delta(\vec \alpha(z))&=
\sum_{S\subseteq[N]} \frac 1{A_{S,z}(\rho)}\,
\vec\al(z_S)\vec\be(z_{S^{c}})\otimes \vec\ga(z_{S^c})\vec\al(z_S)
\\
\label{dg}\Delta(\vec \gamma(z))&=
\sum_{S\subseteq[N]} \frac 1{A_{S,z}(\rho)}\,\vec\ga(z_S)\vec\de(z_{S^c})\otimes \vec\ga(z_{S^c})\vec\al(z_S),
\end{align}
where $A_{s,z}$ is as in \eqref{ph} and
we write $f(\rho)=f(\mu)\ot 1=1\ot f(\la)$.
\end{lemma}

\begin{proof}
By \eqref{bme} and  Corollary \ref{vec},
\begin{equation*}\begin{split}\Delta(\vec \alpha(z))&=\sum_{S\subseteq[N]}F_S(z)\otimes E_S(z)=\sum_{S,T\subseteq[N]}F_S(z)\otimes\mu_l(A_{T,z}^{-1}\lx f_T,E_S\rx 1)e_T(z)\\ 
&=\sum_{S,T\subseteq[N]}\mu_r(A_{T,z}^{-1}\lx f_T,E_S\rx 1)F_S(z)\otimes e_T(z)=\sum_{T\subseteq[N]}\mu_r(A_{T,z}^{-1})f_T(z)\otimes e_T(z),
\end{split}\end{equation*}
which is \eqref{da}.

It is clear from the defining relations that 
$\eta(f(\mu))=f(\mu)$, $\eta(\al(z))=\ga(z)$, $\eta(\be(z))=\de(z)$
extends to  an algebra isomorphism (though not an $\gh$-algebra isomorphism)
between subalgebras of $\E$. It is easy to check that $(\eta\ot\id)\circ\De=\De\circ\eta$. Applying $\eta\otimes\id$ to \eqref{da} then gives \eqref{dg}.
\end{proof}

\section{Elliptic weight functions}
\label{ews}

In contrast to the pairings between $\vec\alpha$ and  $\vec\delta$
considered in Lemma \ref{adl}, the pairing
$\lx\vec\beta(w),\vec\gamma(z)\rx $
is not given by an elementary product. By the discussion in \S \ref{fs}, it can be identified with the partition function of the 8VSOS model  with domain wall boundary conditions, see Figure \ref{dwf}.
It is also a special case of the elliptic weight functions introduced in \cite{tv}, see further \cite{ftv1,ftv2,fv2,mv}.

\begin{figure}[h]
\begin{center}
\begin{pspicture}(0.9,0.9)(3.5,3.5)
\psset{xunit=.7,yunit=.7}
\psset{arrowsize=5pt}
\rput(0.5,2.2){$\vdots$}
\rput(4.5,2.2){$\vdots$}
\rput(3,0.5){$\dotsm$}
\rput(3,4.5){$\dotsm$}
\rput(1,1){\iceL}
\rput(1,3){\iceL}
\rput(1,4){\iceL}
\rput(1,5){\iceD}
\rput(2,5){\iceD}
\rput(4,5){\iceD}
\rput(1,0){\iceU}
\rput(2,0){\iceU}
\rput(4,0){\iceU}
\rput(4,1){\iceR}
\rput(4,3){\iceR}
\rput(4,4){\iceR}
\psline{-}(1,1)(1,4)
\psline{-}(1,4)(4,4)
\psline{-}(4,4)(4,1)
\psline{-}(4,1)(1,1)
\end{pspicture}
\end{center}
\vspace{7mm}
\caption{Domain wall boundary conditions.}
\label{dwf}
\end{figure}
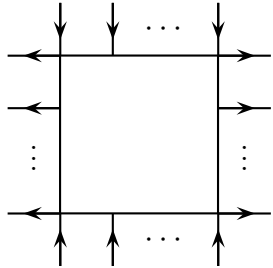

The following identity is essentially obtained in  \cite{tv}, although 
neither the cobraiding nor the relation to the domain wall partition function are  discussed there. Below, we give a simple proof using properties of the cobraiding. For the same identity in the context of the 8VSOS model, see  
\cite{pr,rs}. 
In \cite{rs}, we also obtained an alternative expression, analogous to
 the Izergin--Korepin determinant for the six-vertex model. 

\begin{proposition}\label{bcgp} For generic $(w,z)\in(\mathbb C^\times)^n\times(\mathbb C^\times)^n$, 
$$\lx \vec\be(w),\vec\ga(z)\rx=\frac{\theta(q)^n}{(q^{-\la-n})_n}
\,\Phi(w;z;q^{-\la}),$$
where
\begin{equation}\label{pd}\Phi(w;z;a)\\=
\sum_{\sigma\in S_n}\prod_{1\leq i<j\leq n}
\frac{\theta(qz_{\sigma(i)}/z_{\sigma(j)})\theta(w_i/z_{\sigma(j)})}{\theta(z_{\sigma(i)}/z_{\sigma(j)})\theta(qw_i/z_{\sigma(j)})}\prod_{j=1}^n
\frac{\theta(aq^{-j}w_j/z_{\sigma(j)})}{\theta(qw_j/z_{\sigma(j)})}.
 \end{equation}
\end{proposition}

Note that $\Phi(w;z;a)$ has poles only at $qw_i/z_j\in p^{\mathbb Z}$; the 
singularities at $z_i/z_j\in p^{\mathbb Z}$ cancel in the symmetrization.

\begin{proof}
We  write
$\lx \vec\be(w),\vec\ga(z)\rx=\lx \vec\be(w_S)\vec\be(w_{S^c}),\vec\ga(z)\rx $
for  $S\subseteq [n]$. Applying first \eqref{ppp}
and \eqref{dg},  then \eqref{pob} and \eqref{pod}, we obtain
\begin{multline*}\lx \vec\be(w),\vec\ga(z)\rx=\sum_{T\subseteq [n],\,|T|=m}
\lx \vec\be(w_S),\vec\ga(z_T)\vec\de(z_{T^c})\rx T_{2m-n}A_{T,z}^{-1}
\lx \vec\be(w_{S^c}),\vec\ga(z_{T^c})\vec\al(z_{T^c})\rx\\
=\sum_{T\subseteq [n],\,|T|=m}\,\prod_{i\in S,j\in T^c}\frac{\tha(w_i/z_j)}{\tha(qw_i/z_j)}\prod_{i\in T,j\in T^c}\frac{\tha(qz_i/z_j)}{\tha(z_i/z_j)}\\
\times\lx \vec\be(w_S),\vec\ga(z_T)\rx T_{m}
\lx \vec\be(w_{S^c}),\vec\ga(z_{T^c})\rx T_{-m},
 \end{multline*}
where $m=|S|$.

This is amenable to iteration. For $[n]=S_1\sqcup\dots\sqcup S_N$ (disjoint union), 
\begin{multline}\label{cf}\lx \vec\be(w),\vec\ga(z)\rx
=\sum_{\substack{[n]=T_1\sqcup\dots\sqcup T_N\\|T_i|=|S_i|,\, 1\leq i\leq N}}
\prod_{1\leq k<l\leq N}\left(\prod_{i\in S_k,j\in T_l}\frac{\tha(w_i/z_j)}{\tha(qw_i/z_j)}\prod_{i\in T_k,j\in T_l}\frac{\tha(qz_i/z_j)}{\tha(z_i/z_j)}\right)\\
\times\prod_{j=1}^n\lx\be(w_{S_j}),\ga(z_{T_j})\rx(\la+\textstyle\sum_{k=1}^{j-1}|S_k|).
 \end{multline}
Consider the  case $N=n$,  $S_j=\{j\}$. Then, $T_j=\{\sigma(j)\}$ for some $\sigma\in S_n$, so that
$\lx\be(w_{S_j}),\ga(z_{T_j})\rx=b(\la,w_j/z_{\sigma(j)}).$
 This yields the desired identity.
\end{proof}

Since it may have some independent interest, we rewrite
 \eqref{cf} in terms of $\Phi$.

\begin{corollary}
For any decomposition $[n]=S_1\sqcup\dots\sqcup S_N$, 
\begin{multline*}\Phi(w;z;a)
=\sum_{\substack{[n]=T_1\sqcup\dots\sqcup T_N\\|T_i|=|S_i|,\, 1\leq i\leq N}}
\prod_{1\leq k<l\leq N}\left(\prod_{i\in S_k,j\in T_l}\frac{\tha(w_i/z_j)}{\tha(q w_i/z_j)}\prod_{i\in T_k,j\in T_l}\frac{\tha(qz_i/z_j)}{\tha(z_i/z_j)}\right)\\
\times\prod_{j=1}^n\Phi(w_{S_j};z_{T_j};q^{-\sum_{k=1}^{j-1}|S_k|}a).
 \end{multline*}
\end{corollary}

Choosing $x=\vec\be(w)$ and $y=\vec\ga(z)$ in  \eqref{psu} gives
$$\lx\vec\be(w),\vec\ga(z)\rx(\la)=\lx\vec\ga(w),\vec\be(z)\rx(-\la-2). $$
This proves the following fact.

\begin{corollary}\label{gbc}
For generic $w,z\in(\mathbb C^\times)^n$,
$$\lx \vec\ga(w),\vec\be(z)\rx=\frac{\theta(q)^n}{(q^{\la+2-n})_n}
\Phi(w;z;q^{\la+2}).$$
\end{corollary}

The function $\Phi$ has some  symmetries, which can be explained in terms of symmetries of the algebra $\E$.

\begin{corollary}\label{ewc}
The function $\Phi$ satisfies
\begin{equation*}\begin{split}\Phi(w;z;a)&=\Phi(z^{-1};w^{-1};a)\\
&=q^{-n}a^n\prod_{i,j=1}^n\frac{\tha(w_i/z_j)}{\tha(qw_i/z_j)}
\Phi(w^{-1};qz^{-1};q^{n+2}a^{-1})\\
&=q^{-n}a^n\prod_{i,j=1}^n\frac{\tha(w_i/z_j)}{\tha(qw_i/z_j)}
\Phi(z;qw;q^{n+2}a^{-1}),
\end{split}\end{equation*}
where we use the notation $z^{-1}=(z_1^{-1},\dots,z_n^{-1})$.
\end{corollary}

\begin{proof}
Choosing  $x=\vec\be(w)$ and $y=\vec\ga(z)$ in \eqref{ic} gives
$$\lx\vec\be(w),\vec\ga(z)\rx(\la)=\lx\vec\be(z^{-1}),\vec\ga(w^{-1})\rx(\la).$$
 This shows the equality between the first and second member.
The equality of the  first and third member is a special case of
 Lemma \ref{pfs}. Alternatively, it can be obtained from  \eqref{pd}, replacing $\sigma(i)$ by $\sigma(n+1-i)$ and $w_i$ by $w_{n+1-i}$.  The last equality follows by combining the other two.
\end{proof}

If we specialize $w$ or $z$ 
to a geometric progression, $\Phi$ factors.

\begin{lemma}\label{pgl}
One has
\begin{align*}\Phi(w;z;a)\Big|_{z_j=q^{j-1}\zeta }
&=\frac{(q)_n}{\theta(q)^n}\prod_{j=1}^n\frac{\tha(q^{-n}aw_j/\zeta)}{\tha(q w_j/\zeta)},\\
\Phi(w;z;a)\Big|_{w_j=q^{j-1}\omega }
&=\frac{(q)_n}{\theta(q)^n}\prod_{j=1}^n\frac{\tha(q^{-1}a\om/z_j)}{\tha(q^{n}\omega/z_j)}.
\end{align*}
\end{lemma}

\begin{proof} By symmetry, to prove the first identity we may  put 
$z_j=q^{n-j}\zeta $ in  \eqref{pd}. Then, the sum reduces to the single term with $\sigma=\id$. The second identity  follows using Corollary \ref{ewc}.
\end{proof}

The function $\Phi$ appears
 in the following commutation relations.

\begin{lemma}\label{gag} 
Let $S\subseteq [N]$ with $|S|=m$. Then, for generic $z\in (\mathbb C^\times)^N$,
\begin{equation}\label{agc}\vec\al(z_S)\vec\ga(z_{S^c})=\sum_{T\subseteq [N],\,|T|=m}C_{S,T,z}(\la)\,\vec\ga(z_{T^c})\vec\al(z_T), \end{equation}
where
$$C_{S,T,z}(\la)=\frac{\theta(q)^n(q^{\la+2+n-m})_{m-n}}{(q^{\la+2+N-2m})_m}
\prod_{i\in T,j\in T^c}\frac{\theta(qz_i/z_j)}{\theta(z_i/z_j)}\,
\Phi(z_{T\setminus S};z_{S\setminus T};q^{\la+2+n-m}),$$
with $n=|S\setminus T|=|T\setminus S|$.
In the same notation,
\begin{equation}\label{bac}\vec\be(z_{S^c})\vec\al(z_S)=\sum_{T\subseteq [N],\,|T|=m}D_{S,T,z}(\mu)\,\vec\al(z_T)\vec\be(z_{T^c}), \end{equation}
where
\begin{multline*}D_{S,T,z}(\mu)=\frac{\theta(q)^n(q^{\mu+2-m})_{m}}{(q^{-\mu+m-N})_n(q^{\mu+2+N-2m})_m}\prod_{i\in T,j\in T^c}\frac{\theta(qz_i/z_j)}{\theta(z_i/z_j)}\\
\times\Phi(z_{T\setminus S};z_{S\setminus T};q^{-\mu+m+n-N}).\end{multline*}
\end{lemma}

\begin{proof}
Choosing $x=\vec\al(z_S)\vec\ga(z_{S^c})$ in Corollary \ref{vec} gives
\begin{equation}\label{age}\vec\al(z_S)\vec\ga(z_{S^c})=\sum_{T\subseteq N}
\mu_l\left(A_{T,z}^{-1}\lx\vec\al(z_T)\vec\be(z_{T^c}),\vec\al(z_S)\vec\ga(z_{S^c})\rx 1\right)\vec\ga(z_{T^c})\vec\al(z_T).
 \end{equation}
By Lemma \ref{pool}, Corollary \ref{rl} and  Corollary \ref{gbc}, 
\begin{multline*}\lx\vec\al(z_T)\vec\be(z_{T^c}),\vec\al(z_S)\vec\ga(z_{S^c})\rx\\
\begin{split}&=\lx\vec\al(z_{S\cap T})\vec\al(z_{T\setminus S})\vec\be(z_{S\setminus T})\vec\be(z_{S^c\cap T^c}),\vec\al(z_{S\cap T})\vec\al(z_{S\setminus T})\vec\ga(z_{T\setminus S})\vec\ga(z_{S^c\cap T^c})
\rx
\\
&=T_{-|S\cap T|} \lx\vec\al(z_{T\setminus S})\vec\be(z_{S\setminus T}),\vec\al(z_{S\setminus T})\vec\ga(z_{T\setminus S})\rx T_{-|S\cap T|}\\
&=T_{-|S\cap T|} \lx\vec\ga(z_{T\setminus S}),\vec\be(z_{S\setminus T})\rx T_{-|S\cup T|}\\
&=\delta_{|S|,|T|}\frac{\theta(q)^n}{(q^{\la+2-m})_n}\,
\Phi(z_{T\setminus S};z_{S\setminus T};q^{\la+2+n-m})\,T_{-2m}.
\end{split}\end{multline*}
Plugging this into \eqref{age} yields \eqref{agc}.
The identity  \eqref{bac} can  be proved similarly, or be derived from \eqref{agc} by applying $S\circ\psi$, with $\psi$ as in Proposition \ref{ip}.
\end{proof}

\section{Generalized elliptic $6j$-symbols}
\label{sss}

\subsection{Definition and main properties}\label{dmps}

In \S \ref{ecs} we observed that  $E_S(z)$ form a basis
for the corepresentation $V(z)$, and that $\vec L_{ST}(z)$ are the corresponding matrix elements.
As we have seen in \S \ref{crs}--\ref{fs}, the cobraidings
 $\langle\vec  L_{ST}(w),\vec L_{UV}(z)\rangle$ give a
dynamical $R$-matrix, which can be identified with the partition function for the 8VSOS model with fixed boundary conditions.

We are interested in the dynamical $R$-matrix  corresponding to the alternative basis $e_S(z)$. For generic $z\in(\Cp)^N$, define matrix elements $M_{ST}(z)$ by
\begin{equation}\label{md}\Delta(e_S(z))=\sum_{T\subseteq [N]}M_{ST}(z)\ot e_T(z). \end{equation}
For generic $(w,z)\in(\Cp)^M\times(\Cp)^N$,
we then write
\begin{equation}\label{mrd}\lx M_{ST}(w),M_{UV}(z)\rx
=\mathcal R_{SU}^{TV}(\la;w;z)\, T_{M+N-2|S|-2|U|}.
\end{equation}
 We will refer to $\mathcal R_{SU}^{TV}$  as
a \emph{generalized  $6j$-symbol}. Note that it vanishes unless
 $|S|+|U|= |T|+|V|$. 
By \eqref{tpi} and \eqref{tpd},  $\mathcal R_{SU}^{TV}$ is a matrix element of the natural intertwiner between the corepresentations \eqref{gi}. 

Although it is initially defined for generic values of $(w,z)$, 
 $\mathcal R_{SU}^{TV}$ extends to non-generic values by analytic continuation, and we are particularly interested in such degenerations. For instance, 
when $w_j=q^{j-1}\om$ and $z_j=q^{j-1}\zeta$, 
it reduces to  the  elliptic $6j$-symbols of Date et al.\ \cite{d}.
To understand this, let 
 $V_N(\zeta)$ denote the right-hand side of 
  \eqref{vn} for $z_i=q^{i-1}\zeta$. Using that
$\ga(z)\al(qz)=\al(z)\ga(qz)$,
one finds that  $\dim V_N(\zeta)=N+1$. A basis for  $V_N(\zeta)$ is $(v_s^N(\zeta))_{s=0}^N$, where $v_s^N(\zeta)=E_{[N-s+1,N]}(z)=e_{[N-s+1,N]}(z)$.
One can then introduce  matrix elements  $M_{st}^N(\zeta)$ by
$$\Delta(v_s^N(\zeta))=\sum_{t=0}^N M_{st}^N(\zeta)\otimes v_t^N(\zeta). $$
In \cite{kn},  the pairing $\lx M_{st}^M(\omega),M_{uv}^N(\zeta)\rx$ was expressed as  an elliptic hypergeometric function, which can be identified with an elliptic $6j$-symbol.
We will find analogous  formulas for the  
more general pairing
$\lx M_{ST}(w),M_{UV}(z)\rx$.
Our approach is  different from that of
 \cite{kn}; in  particular, we do not need any explicit expression for the matrix elements. Instead, we  make a more efficient use of  formal properties of the cobraiding.

As was explained in \S \ref{crs}, the symbol
 $\mathcal R_{SU}^{TV}$ satisfies versions of the  QDYB equation  and the
unitarity relation. It seems worth stating these fundamental properties explicitly. 

\begin{proposition}\label{yop}
For $u\in(\Cp)^L$, $w\in(\Cp)^M$, $z\in(\Cp)^N$,  $Q,R\subseteq [L]$, $S,T\subseteq[M]$
and $U,V\subseteq[N]$ with $|Q|+|S|+|U|=|R|+|T|+|V|$,
\begin{multline*}\sum_{\substack{X\subseteq [L],\,Y\subseteq[M],\,Z\subseteq [N]\\|X|+|Y|=|R|+|T|\\|Y|+|Z|=|S|+|U|}}\mathcal  R^{XY}_{RT}(\la+N-2|V|;u,w)
\,\mathcal R^{QZ}_{XV}(\la;u,z)\,\mathcal R^{SU}_{YZ}(\la+L-2|Q|;w,z)\\
=\sum_{\substack{X\subseteq [L],\,Y\subseteq[M],\,Z\subseteq [N]\\|X|+|Y|=|Q|+|S|\\|Y|+|Z|=|T|+|V|}}\mathcal R^{YZ}_{TV}(\la;w,z) \,\mathcal R^{XU}_{RZ}(\la+M-2|Y|;u,z)\,\mathcal R^{QS}_{XY}(\la;u,w).
\end{multline*}
Moreover, for  $w\in(\Cp)^M$, $z\in(\Cp)^N$, $S,T\subseteq[M]$
and $U,V\subseteq[N]$ with $|S|+|U|=|T|+|V|$,
\begin{equation}\label{gbo}\sum_{\substack{X\subseteq[M],\,Y\subseteq[N]\\|X|+|Y|=|S|+|U|}}\mathcal R^{XY}_{SU}(\la;w,z)\,\mathcal R_{YX}^{VT}(\la;z,w)=\delta_{ST}\delta_{UV}.\end{equation}
\end{proposition}

\subsection{An explicit formula and further properties}

Our first main result is the following expression for generalized $6j$-symbols.

\begin{theorem}\label{mcmt}
One has
\begin{multline}\label{mti}\mathcal R_{SU}^{TV}(\la;w;z)
=\frac{(q^{\la+2+M+N-2L})_{|S|}}{(q^{\la+2+M-2|T|})_{|T|}(q^{\la+2+M+N-2L})_{|V|}}\\
\begin{split}&\times\sum_{\substack{X\subseteq S^c\cap T^c, Y\subseteq U\cap V\\ |Y|-|X|=L-M}}\theta(q)^{|U|+|V|-2|Y|}\frac{(q^{\la+2+N-|U|-|Y|})_{|Y|}}{(q^{-\la+2|T|-M})_{|V|-|Y|}}\prod_{i\in T,j\in T^c\setminus X}\frac{\tha(qw_i/w_j)}{\tha(w_i/w_j)}\\
&\times
\prod_{i\in S^c\setminus X,j\in X}\frac{\tha(qw_i/w_j)}{\tha(w_i/w_j)}
\prod_{i\in V\setminus Y,j\in V^c}\frac{\tha(qz_i/z_j)}{\tha(z_i/z_j)}
\prod_{i\in Y,j\in U\setminus Y}\frac{\tha(qz_i/z_j)}{\tha(z_i/z_j)}\\
&\times\prod_{i\in Y,j\in X}\frac{\tha(qz_i/w_j)}{\tha(z_i/w_j)}
\prod_{i\in X^c,j\in Y^c}\frac{\tha(w_i/z_j)}{\tha(qw_i/z_j)}\\
& \times 
{\prod_{i\in S^c\setminus X,j\in U\setminus Y}\frac{\tha(qw_i/z_j)}{\tha(w_i/z_j)}\prod_{i\in T^c\setminus X,j\in V\setminus Y}\frac{\tha(qw_i/z_j)}{\tha(w_i/z_j)}}\\
&\times \Phi(w_{S^c\setminus X};z_{U\setminus Y};q^{\la+2+N-|U|-|Y|})\,
\Phi(w_{T^c\setminus X};z_{V\setminus Y};q^{-\la+|T|-|X|}),
\end{split}\end{multline}
where  $w\in(\Cp)^M$, $z\in(\Cp)^N$ and
\begin{equation}\label{l}L=|S|+|U|=|T|+|V|.\end{equation} 
\end{theorem}

Before proving
Theorem \ref{mcmt}, we discuss some interesting consequences.
 First of all,  using also  Corollary \ref{ewc}, 
it is straight-forward to deduce the following symmetries.

\begin{corollary}\label{cbsl} In the notation above,
\begin{multline*}\mathcal R_{SU}^{TV}(\la;w;z)=\mathcal R_{U^cS^c}^{V^cT^c}(\la+M+N-2L;z^{-1};w^{-1})\\
=\frac{G_{U,z}(\la+N-2|U|)G_{S,w}(\la+M+N-2L)}{G_{T,w}(\la+M-2|T|)G_{V,z}(\la+M+N-2L)}\,
\mathcal R_{V^cT^c}^{U^cS^c}(-\la-2;z^{-1};w^{-1})\\
=\frac{G_{U,z}(\la+N-2|U|)G_{S,w}(\la+M+N-2L)}{G_{T,w}(\la+M-2|T|)G_{V,z}(\la+M+N-2L)}\,
\mathcal R_{TV}^{SU}(-\la-2+2L-M-N;w;z),
\end{multline*}
where 
\begin{equation}\label{gsd}\begin{split}G_{S,z}(\la)&=(-1)^{|S|}q^{-\binom{|S|}2} q^{-\frac 12 N\la}(q^{\la+1+|S|-N})_{|S|}(q^{\la+2+2|S|-N})_{N-|S|}\\
&\quad\times\prod_{i\in S,j\in S^c}\frac{\tha(z_i/z_j)}{\tha(qz_i/z_j)}. \end{split}\end{equation} 
\end{corollary}

 One can give a more instructive proof of Corollary \ref{cbsl} using the
 algebra symmetries of Proposition \ref{ip}.
For instance, one has
$$e_S(z)^\ast=\vec\de(z_S^{-1})\vec\be(z_{S^c}^{-1})=M_{S^c,\emptyset}(z^{-1}), $$
where the second equality follows from Lemma \ref{dac}.  Applying
$\De$ to this equality, using \eqref{crm} and
$(\ast\ot\ast)\circ\De=\De\circ\ast $,
gives
$M_{ST}(z)^\ast=M_{S^cT^c}(z^{-1})$.
Choosing $x=M_{ST}(w)$ and  $y=M_{UV}(z)$ in \eqref{ic} then 
yields the first equality in Proposition~\ref{cbsl}. 
Similarly, it follows from \eqref{sm} below that
$$(\phi\circ S)(M_{ST}(z))=\frac{G_{S,z}(-\mu-2+N-2|S|)}{G_{T,z}(-\la-2+N-2|T|)}\,\overrightarrow{{\det}}^{-1}(z^{-1})M_{T^c S^c}(qz^{-1}).$$
(Here and in several places below, the letter $S$ is used both for the antipode and for a set; we hope this will not confuse the reader.)
Then, the identity
$$\lx M_{ST}(w),M_{UV}(z)\rx=\psi^{\D}\big(\lx (\phi\circ S)(M_{UV}(z))  , (\phi\circ S)( M_{ST}(w)) \rx\big) $$
 leads to the equality of the first and third member (for the computation, one needs Lemma  \eqref{dcgl}).
The remaining symmetry follows by combining the other two.

In special situations, the expression \eqref{mti} simplifies.

\begin{corollary}
If any one of the four conditions $|V|<|S\setminus T|$,  $ |U|<|T\setminus S|$,
$ |S^c|<|U\setminus V|$ or $ |T^c|<|V\setminus U|$ holds, then $\mathcal R_{SU}^{TV}$ vanishes identically. If either $|V|=|S\setminus T|$,  $ |U|=|T\setminus S|$,
$ |S^c|=|U\setminus V|$ or $ |T^c|=|V\setminus U|$, then $\mathcal R_{SU}^{TV}$ is given by an elementary factor times a product of two elliptic weight functions. Finally, if either $S^c=\emptyset$, $T^c=\emptyset$, $U=\emptyset$ or $V=\emptyset$, then $\mathcal R_{SU}^{TV}$ is given by an elementary factor times a single elliptic weight function.
\end{corollary}

\begin{proof} 
With $X$ and $Y$  as in \eqref{mti}, 
 $$|S^c\cap T^c|\geq |X|=|Y|+|S^c|-|U|\geq |S^c|-|U|,$$
$$|U\cap V|\geq |Y|=|U|-|S^c|+|X|\geq |U|-|S^c|.  $$
Thus, if $ |U|<|T\setminus S|$ or
$ |S^c|<|U\setminus V|$, the sum is empty.
 Since  $|S|+|U|=|T|+|V|$, 
the remaining part of the first statement follows. By the same argument, 
if any one of the eight equalities stated hold, the sum has only one term.
\end{proof}

As an example, to be used later,
\begin{equation}\label{rese}\begin{split}\mathcal R_{SU}^{T\emptyset}(\la;w;z)&={\theta(q)^{|U|}}\prod_{i\in T\setminus S,j\in U}\frac {\tha(qw_i/z_j)}{\tha(w_i/z_j)}\prod_{i\in T\setminus S,j\in T^c}
\frac{\tha(qw_i/w_j)}{\tha(w_i/w_j)}\prod_{i\in T,j\in [N]}
\frac{\tha(w_i/z_j)}{\tha(qw_i/z_j)}\\
&\quad\times\frac{(q^{\la+2+M+N-2|T|})_{|S|}}{(q^{\la+2+M-2|T|})_{|T|}}\,\Phi(w_{T\setminus S};z_U;q^{\la+2+N-|U|})
\end{split}\end{equation}
if $S\subseteq T$ with $|T\setminus S|=|U|$, and vanishes else.

Finally, the following fact is needed in \S \ref{bos}.

\begin{corollary}\label{qvs} Suppose there exists $(i,j)\in T\times T^c $ with $w_j=qw_i$. Then, either $(i,j)\in S\times S^c$ or $\mathcal R_{SU}^{TV}(\la;w;z)$ vanishes identically. Similarly, if $(i,j)\in V\times V^c $ and $z_j=qz_i$, then either $(i,j)\in U\times U^c$ or $\mathcal R_{SU}^{TV}(\la;w;z)$ vanishes.
\end{corollary}

\begin{proof}
Consider the expression \eqref{mti}. In the first situation, if $j\in T^c\setminus X$ then the factor $\prod_{i\in T,j\in T^c\setminus X }\tha(qw_i/w_j)$ vanishes. Thus, we may assume $j\in X$; in particular, $j\in S^c$. Then, if $i\in S^c\setminus X$ the factor  $\prod_{i\in S^c\setminus X,j\in X}\tha(qw_i/w_j)$ vanishes.
Thus, non-vanishing terms   exist only when $i\in S$ ($i\in X$ would contradict $i\in T$) and $j\in S^c$.
The proof of   the second statement is similar. 
\end{proof}

\subsection{Proof of Theorem {\ref{mcmt}}}
\label{mps}

The key to the proof of Theorem \ref{mcmt} is the following result.

\begin{proposition}\label{mil}
Let $(w,z)\in(\mathbb C^\times)^M\times (\mathbb C^\times)^N$ be generic, and let $a\in W(w)$, $b\in W(z)$, $c\in V(z)$ and $d\in V(w)$. Then, 
\begin{equation}\label{mi}\lx ab,cd\rx=\sum_{(c)(d)}\lx d',c'\rx T_{\om_{12}(c)+\om_{12}(d)}
\lx a,d''\rx T_M\lx b,c''\rx T_{-M}.
 \end{equation}
\end{proposition}

\begin{proof}
We may choose
$a=f(\mu)F_T(w)$, $b=g(\mu)F_V(z)$, $c=h(\la) E_U(z)$, $d=k(\la)E_S(w)$, where $f,g,h,k\in\mg$.
It is easy to check that the 
functions $f$, $g$, $h$, $k$ cancel from both sides of \eqref{mi}, and we
are reduced to proving 
\begin{multline*}\lx \vec L_{+T}(w)\vec L_{+V}(z),\vec L_{U+}(z)\vec L_{S+}(w)\rx
=
\sum_{X\subseteq [M],Y\subseteq[N]}\lx\vec  L_{SX}(w),\vec L_{UY}(z)\rx
T_{|X|+|Y|}\\
\times\lx\vec  L_{+T}(w),\vec L_{X+}(w)\rx T_M \lx\vec  L_{+V}(z),\vec L_{Y+}(z)\rx T_{-M}.
 \end{multline*}
By Proposition \ref{lll}, the right-hand side equals
$\lx\vec L_{ST}(w),\vec L_{UV}(z)\rx T_{-M-N}$.
Thus, the result follows from Corollary \ref{rl}.
\end{proof}

The following result is a transformed version of  Corollary \ref{rl}, where  basis vectors of the form $E_S$ and $F_S$ have been replaced by  $e_S$ and $f_S$. Just as for Corollary \ref{rl}, the right-hand side is a partition function for a square with two domain walls.

\begin{corollary}\label{mmc}
In the notation above,
\begin{multline*} \mathcal R_{SU}^{TV}(\la;w;z)
=\frac{1}{A_{T,w}(\la)A_{V,z}(\la+M-2|T|)}\\
\times\lx\vec  \al(w_T)\vec  \be(w_{T^c})\vec  \al(z_V)\vec  \be(z_{V^c})
,\vec  \ga(z_{U^c})\vec  \al(z_U)\vec  \ga(w_{S^c})\vec  \al(w_S)\rx 1
.
\end{multline*}
\end{corollary}

\begin{proof}
Choose $a=f_T(w)$, $b=f_V(z)$, $c=e_U(z)$ and $d=e_S(w)$ in Lemma \ref{mil}. 
Using \eqref{feo}, we obtain
\begin{multline*}
\lx f_T(w)f_V(z),e_U(z)e_S(w)\rx\\
\begin{split}&=\sum_{X\subseteq[M],Y\subseteq[N]}
\lx M_{SX}(w),M_{UY}(z)\rx T_{2|X|+2|Y|-M-N}\\
&\quad\times\lx f_T(w),e_X(w)\rx T_M
\lx f_V(z),e_Y(z)\rx T_{-M}\\
&=\lx M_{ST}(w),M_{UV}(z)\rx T_{2|T|+2|V|-M-N}A_{T,w}T_{M-2|T|}
A_{V,z}T_{-M-2|V|},
\end{split}\end{multline*}
which gives the desired result after simplification.
\end{proof}

\begin{proof}[Proof of \emph{Theorem \ref{mcmt}}] In Corollary \ref{mmc},  apply \eqref{bac} to the factor $\vec \be(w_{T^c})\vec  \al(z_V)$
and \eqref{agc} to the factor $\vec  \al(z_U)\vec  \ga(w_{S^c})$.
Each resulting  term is of the form \eqref{feo}, with $z$ replaced by $(w,z)$.
The only non-vanishing terms are those where
$\vec \be(w_{T^c})\vec  \al(z_V)$  is replaced (up to a multiplier) by $\vec\al(w_{T^c\setminus X},z_Y)\vec\be(w_X,z_{V\setminus Y})$
 and simultaneously   $\vec  \al(z_U)\vec  \ga(w_{S^c})$ by 
$\vec\ga(w_X,z_{U\setminus Y})\vec\al(w_{S^c\setminus X},z_{Y})$, for some 
$X\subseteq S^c\cap T^c$, $Y\subseteq U\cap V$. 
 All in all, this gives
\begin{multline*}
\mathcal R_{SU}^{TV}(\la;w;z)=\sum_{\substack{X\subseteq S^c\cap T^c, Y\subseteq U\cap V\\ |Y|-|X|=L-M}}\frac{A_{(X^c,Y),(w,z)}(\la)}{A_{T,w}(\la)A_{V,z}(\la+M-2|T|)}\\
\times C_{(\emptyset,U),(S^c\setminus X,Y),(w_{S^c},z_U)}(\la+N-|U|)D_{(\emptyset,V),(T^c\setminus X,Y),(w_{T^c},z_V)}(\la-|T|),
\end{multline*}
which yields \eqref{mti} after simplification. 
\end{proof}

\subsection{Asymmetric identities}

We have seen that \eqref{mti} displays all  symmetries of generalized $6j$-symbols  given in Corollary \ref{cbsl}.  From the viewpoint of special functions, it is interesting to obtain \emph{less} symmetric expressions, since symmetries then correspond to non-trivial transformation formulas. 

Recall the symmetry  $\phi$ defined in 
 Proposition \ref{ip}. Clearly, $\phi$ restricts to an $\gh$-space isomorphism $V(z^{-1})\rightarrow V(z)$. In particular,
$$\phi(e_S(z^{-1}))=\vec\al(z_S)\vec\ga(z_{S^c}),\qquad S\subseteq [N], $$
form a basis for $V(z)$. Similarly, 
$$\phi(f_S(z^{-1}))=\vec\be(z_{S^c})\vec\al(z_{S}),\qquad S\subseteq [N], $$
form a basis of $W(z)$. 

The following lemma  can be proved similarly as \eqref{feo}.

\begin{lemma}
For generic $z\in(\mathbb C^\times)^N$, 
$$\lx\phi(f_T(z^{-1})),\phi(e_S(z^{-1}))\rx=
\delta_{ST}B_{S,z} T_{-2m},$$
where $m=|S|$ and where
\begin{equation}\label{b}B_{S,z}(\la) = q^{m(N-m)}
\frac{(q^{\la+1-m})_{N-m}}{(q^{\la+1})_{N-m}}
\prod_{i\in S^c,j\in S}\frac{\tha(z_i/z_j)}{\tha(qz_i/z_j)}.\end{equation}
\end{lemma}

The following result is then obtained 
similarly as 
Corollary \ref{mmc}.

\begin{lemma}\label{pmmc}
Let $(w,z)\in(\mathbb C^\times)^M\times(\mathbb C^\times)^N$ be generic. Then, for $S,T\subseteq [M]$ and $U,V\subseteq [N]$,
\begin{multline*}\lx \phi(M_{ST}(w^{-1})),M_{UV}(z)\rx\\
=
\frac{1}{B_{T,w}(\la)A_{V,z}(\la+M-2|T|)}\,\lx \phi(f_T(w^{-1}))f_V(z),e_U(z)\phi(e_S(w^{-1}))\rx T_{M+N}.\end{multline*}
\end{lemma}

We need to relate the action of $\phi$ and $S$, first on basis vectors and then on matrix elements.

\begin{lemma}\label{sel}
One has
$$S(e_S(z))=(-1)^{N}
\frac{q^{\frac 12N\la+\binom N2}}{(q^{\la+1})_N}\,G_{S,z}(\mu)
\overrightarrow{{\det}}^{-1}(q^{-1}z)\phi(M_{\emptyset S^c}(qz^{-1})), 
$$
where  $G$ is as in \eqref{gsd}.
\end{lemma}

\begin{proof}
Using \eqref{sv}, we can write
\begin{equation*}\begin{split}S(e_S(z))&= (-1)^{N-m}\frac{\prod_{j=1}^mF(\mu+j-1)\prod_{j=1}^{N-m}F(\mu+2m-N+j)}{\prod_{j=1}^NF(\la+j-1)}\\
&\quad\times\overrightarrow{{\det}}^{-1}(q^{-1}z)\vec\de(q^{-1}z_S)\vec\ga(q^{-1}z_{S^c})\\
&= (-1)^{N-m}q^{\frac 12N(\la-\mu)+\binom {N-m}2}
\frac{(q^{\mu+1})_m(q^{\mu+2+2m-N})_{N-m}}{(q^{\la+1})_N}\\
&\quad\times\overrightarrow{{\det}}^{-1}(q^{-1}z)\vec\de(q^{-1}z_S)\vec\ga(q^{-1}z_{S^c}),
\end{split}\end{equation*}
where $m=|S|$.
On the other hand, it follows from \eqref{dg} that
$$M_{\emptyset S}(z)=\frac{1}{A_{S,z}(\mu)}\,\vec\ga(z_S)\vec\de(z_{S^c}), $$
which leads to
\begin{equation*}\begin{split}\phi(M_{\emptyset S^c}(qz^{-1}))&=\frac{1}{A_{S^c,qz^{-1}}(-\mu-2+N-2m)}\,\vec\de(q^{-1}z_{S})\vec\ga(q^{-1}z_{S^c})\\
&=q^{m(m-N)}\frac{(q^{\mu+1})_m}{(q^{\mu+1+m-N})_m}\prod_{i\in S,j\in S^c}\frac{\tha(qz_i/z_j)}{\tha(z_i/z_j)}\,\vec\de(q^{-1}z_S)\vec\ga(q^{-1}z_{S^c}).
\end{split}\end{equation*}
Combining these facts yields  the desired result.
\end{proof}

\begin{corollary}\label{spp}
For  $z\in(\mathbb C^\times)^N$,
\begin{equation}\label{sm}S(M_{ST}(z))=\frac{G_{S,z}(\mu)}{G_{T,z}(\la)}\,\overrightarrow{{\det}}^{-1}(q^{-1}z)\phi(M_{T^c S^c}(qz^{-1})).\end{equation}
\end{corollary}

\begin{proof}
By \eqref{aprop}, Lemma \ref{sel} and \eqref{dph},
\begin{multline*}\sum_{T\subseteq[N]}S(e_T(z))
\ot S(M_{ST}(z))=\Delta(S(e_S(z)))\\
=\sum_{T\subseteq[N]}
S(e_T(z))
\ot \frac{G_{S,z}(\mu)}{G_{T,z}(\la)}\,\overrightarrow{{\det}}^{-1}(q^{-1}z)\phi(M_{T^c S^c}(qz^{-1})).
 \end{multline*}
Since, by Corollary \ref{vec}, the elements $e_T(z)$ are linearly independent over $\mu_l(\mg)$ and $S$ is invertible, $S(e_T(z))$ are  linearly independent over $\mu_r(\mg)$. It follows that the identity above holds termwise. 
\end{proof}

We can now obtain the following variation of Corollary \ref{mmc}.

\begin{corollary}\label{mma}
Let $(w,z)\in(\mathbb C^\times)^M\times(\mathbb C^\times)^N$ be generic. Then,
\begin{multline*}\mathcal R_{SU}^{TV}(\la;w;z)=(-1)^{|U|+|V|}q^{\binom{|V|}2+N(M-L)+\frac 12|U|(|U|+1)}\\
\begin{split}&\times
\prod_{i\in T,j\in T^c}\frac{\tha(qw_i/w_j)}{\tha(w_i/w_j)}
\prod_{i\in V,j\in V^c}\frac{\tha(qz_i/z_j)}{\tha(z_i/z_j)}
\prod_{i\in[M],j\in[N]}\frac{\tha(w_i/z_j)}{\tha(qw_i/z_j)}\\
&\times \frac{(q^{\la+2-|T|})_{N+|T|-|U|}(q^{\la+1+N-2|U|})_{|U|}}{(q^{\la+1+M-2|T|-|V|})_{|V|}(q^{\la+2+M-2|T|})_{|T|}(q^{\la+2+M-2|T|})_{N-|V|}}\, T_{N-2|U|}\\
&\times\lx \vec\be(q^{-1}z_{U})\vec\al(q^{-1}z_{U^c},w_T)\vec\be(w_{T^c}),
\vec\ga(w_{S^c})\vec\al(w_S,q^{-1}z_{V^c})\vec\ga(q^{-1}z_V)\rx 1,\end{split}\end{multline*}
with $L$ as in \eqref{l}.
\end{corollary}

\begin{proof}
Using, respectively, Proposition \ref{acl}, Corollary \ref{spp}, Lemma \ref{dcgl} and  Lemma~\ref{pmmc} gives
\begin{multline*} \lx M_{ST}(w),M_{UV}(z)\rx=T_{N-2|U|} \lx S(M_{UV}(z)),M_{ST}(w)\rx T_{N-2|V|}\\
\begin{split}&=T_{N-2|U|}\lx\frac{G_{U,z}(\mu)}{G_{V,z}(\la)}\,\overrightarrow{{\det}}^{-1}(q^{-1}z)\phi(M_{V^c U^c}(qz^{-1})),M_{ST}(w)\rx T_{N-2|V|}\\
&=\frac{G_{U,z}(\la+N-2|U|)}{G_{V,z}(\la+M+N-2|T|-2|V|)}q^{\frac 12MN}\prod_{i\in[M],j\in[N]}\frac{\tha(w_i/z_j)}{\tha(qw_i/z_j)}\\
&\quad\times T_{N-2|U|}\lx\phi(M_{V^c U^c}(qz^{-1})),M_{ST}(w)\rx T_{N-2|V|}\\
&=q^{\frac 12 MN}\prod_{i\in[M],j\in[N]}\frac{\tha(w_i/z_j)}{\tha(qw_i/z_j)}\\
&\quad\times
\frac{G_{U,z}(\la+N-2|U|)}{G_{V,z}(\la+M+N-2|T|-2|V|)A_{T,w}(\la)B_{U^c,q^{-1}z}(\la+N-2|U|)}\\
&\quad\times T_{N-2|U|}
\lx \phi(f_{U^c}(qz^{-1}))f_T(w) ,e_S(w)\phi(e_{V^c}(qz^{-1}))\rx
 T_{M+2N-2|V|},
\end{split}\end{multline*}
which simplifies to the given expression.
\end{proof}

\begin{theorem}\label{rat} In the notation above,
 $\mathcal R_{SU}^{TV}(\la;w;z)$  can be expressed as
\begin{multline*}(-1)^{|U|+|V|}q^{\binom{|V|}2+N(M-L)+\frac 12|U|(|U|+1)}\prod_{i\in V,j\in V^c}\frac{\tha(qz_i/z_j)}{\tha(z_i/z_j)}\prod_{i\in[M],j\in[N]}\frac{\tha(w_i/z_j)}{\tha(qw_i/z_j)}\\
\begin{split}&\times \frac{(q^{\la+1+N-2|U|})_{|U|}}{(q^{\la+1+M-2|T|-|V|})_{|V|}(q^{\la+2+M-2|T|})_{|T|}(q^{\la+2+M-2|T|})_{N-|V|}}\\
&\times \sum_{\substack{X\subseteq S\cap T,\,Y\subseteq[N]\\|X|+|Y|=|S|+N-|V|}}
\tha(q)^{|U\cap Y|+|V\cap Y|}\prod_{i\in X,j\in S\setminus X}\frac{\tha(qw_i/w_j)}{\tha(w_i/w_j)}
\prod_{i\in T\setminus X,j\in T^c}\frac{\tha(qw_i/w_j)}{\tha(w_i/w_j)}
\\
&\times
\prod_{i\in X,j\in Y^c}\frac{\tha(q^2w_i/z_j)}{\tha(qw_i/z_j)}\prod_{i\in S^c\cap T^c,j\in Y}\frac{\tha(z_j/qw_i)}{\tha(z_j/w_i)}
\prod_{i\in S\cap T\cap X^c,j\in Y}\frac{\tha(z_j/w_i)}{\tha(z_j/qw_i)}
\\
&\times\prod_{i\in Y,j\in Y^c}\frac {\tha(qz_i/z_j)}{\tha(z_i/z_j)
}
\frac{(q^{\la+2+M-2|T|-|V|+|V\cap Y|})_{N+|S|-|V|-|V\cap Y|}}{(q^{-\la-N+|U|})_{|U\cap Y|}}\\
&\times\Phi(q^{-1}z_{ U\cap Y};(w_{T\setminus X},q^{-1}z_{U^c\cap Y^c});q^{-\la-N+|U|+|U\cap Y|})\\
&\times\Phi(q^{-1}z_{ V\cap Y};(w_{S\setminus X},q^{-1}z_{V^c\cap Y^c});q^{\la+2+M-2|T|-|V|+|V\cap Y|})\end{split}\end{multline*}
or alternatively as
\begin{multline*}(-1)^{|S|+|T|}q^{\binom{|S|}2+M(|U|-|T|)+\frac 12|T|(|T|+1)}\prod_{i\in T,j\in T^c}\frac{\tha(qw_i/w_j)}{\tha(w_i/w_j)}\prod_{i\in[M],j\in[N]}\frac{\tha(w_i/z_j)}{\tha(qw_i/z_j)}\\
\begin{split}&\times \frac{(q^{\la+1+N-2|U|})_{M-|S|}}{(q^{\la+1-|T|})_{M-|T|}(q^{\la+2+M-2|T|})_{|T|}(q^{\la+2+M-2|T|})_{N-|V|}}\\
&\times \sum_{\substack{X\subseteq U^c\cap V^c,\,Y\subseteq[M]\\|X|+|Y|=N+|S|-|V|}}
\tha(q)^{|S^c\cap Y|+|T^c\cap Y|}\prod_{i\in U^c\setminus X,j\in X}\frac{\tha(qz_i/z_j)}{\tha(z_i/z_j)}
\prod_{i\in V,j\in V^c\setminus X}\frac{\tha(qz_i/z_j)}{\tha(z_i/z_j)}
\\
&\times
\prod_{i\in Y^c,j\in X}\frac{\tha(q^2w_i/z_j)}{\tha(qw_i/z_j)}\prod_{i\in Y,j\in U\cap V}\frac{\tha(z_j/qw_i)}{\tha(z_j/w_i)}
\prod_{i\in Y,j\in U^c\cap V^c\cap X^c}\frac{\tha(z_j/w_i)}{\tha(z_j/qw_i)}
\\
&\times\prod_{i\in Y^c,j\in Y}\frac {\tha(qw_i/w_j)}{\tha(w_i/w_j)
}
\frac{(q^{\la+2-|T|+|T^c\cap Y|})_{N+|S|-|V|-|T^c\cap Y|}}{(q^{-\la-M-N+|S|+2|U|})_{|S^c\cap Y|}}\\
&\times\Phi((z_{V^c\setminus X},qw_{S\cap Y^c});qw_{S^c\cap Y};q^{-\la-M-N+|S|+2|U|+|S^c \cap Y|})\\
&\times\Phi((z_{U^c\setminus X},qw_{T\cap Y^c});qw_{T^c\cap Y};q^{\la+2-|T|+|T^c\cap Y|})
.\end{split}\end{multline*}
\end{theorem}

\begin{proof}
Consider the factor 
$$\lx \vec\be(q^{-1}z_{U})\vec\al(q^{-1}z_{U^c},w_T)\vec\be(w_{T^c}),
\vec\ga(w_{S^c})\vec\al(w_S,q^{-1}z_{V^c})\vec\ga(q^{-1}z_V)\rx$$
from Corollary \ref{mma}.
Similarly as in the proof of Theorem \ref{mcmt}, commuting
$\beta$ to the right and $\gamma$ to the left, 
it can be written
\begin{multline*}\sum_{\substack{X\subseteq S\cap T,\,Y\subseteq[N]\\|X|+|Y|=|S|+N-|V|}}
A_{(X,Y),(w,q^{-1}z)}(\la)\\
\times C_{(S,V^c),(X,Y),(w_S,q^{-1}z)}(\la+M-|S|)D_{(T,U^c),(X,Y),(w_T,q^{-1}z)}(\la)T_{2|V|-2|S|-2N}.
\end{multline*}
Inserting this into Lemma \ref{mma} and simplifying, we obtain the first expression. The second expression then follows using the first symmetry of Proposition \ref{cbsl}.
\end{proof}

An interesting summation formula follows by choosing $V=\emptyset$ in the second expression of Theorem \ref{rat}, the value of the sum being known from \eqref{rese}. From the results of \S  \ref{hss}, it will be clear that this identity generalizes the elliptic $A_n$ Jackson summation of \cite{rr}; see \eqref{rjs} below.

\begin{corollary} Suppose that $S,T\subseteq [M]$ and $U\subseteq [N]$ with 
$|S|+|U|=|T|$. Then, if $S\subseteq T$, 
\begin{multline*} \sum_{\substack{X\subseteq U^c,\,Y\subseteq[M]\\|X|+|Y|=N+|S|}}
\tha(q)^{|S^c\cap Y|+|T^c\cap Y|}\prod_{i\in U^c\setminus X,j\in X}\frac{\tha(qz_i/z_j)}{\tha(z_i/z_j)}\prod_{i\in Y^c,j\in Y}\frac {\tha(qw_i/w_j)}{\tha(w_i/w_j)
}
\\
\begin{split}&\quad\times\prod_{i\in Y^c,j\in X}\frac{\tha(q^2w_i/z_j)}{\tha(qw_i/z_j)}\prod_{i\in Y,j\in U^c\cap X^c}\frac{\tha(z_j/w_i)}{\tha(z_j/qw_i)}
\frac{(q^{\la+2-|T|+|T^c\cap Y|})_{N+|S|-|T^c\cap Y|}}{(q^{-\la-M-N+|T|+|U|})_{|S^c\cap Y|}}\\
&\quad\times\Phi((z_{X^c},qw_{S\cap Y^c});qw_{S^c\cap Y};q^{-\la-M-N+|T|+|U|+|S^c \cap Y|})\\
&\quad\times\Phi((z_{U^c\setminus X},qw_{T\cap Y^c});qw_{T^c\cap Y};q^{\la+2-|T|+|T^c\cap Y|})\\
&=(-1)^{|U|}q^{(M-|T|)|S|-\frac 12|U|(|U|+1)}\theta(q)^{|U|}\frac{(q^{\la+1-|T|})_{M-|T|}(q^{\la+2+M-2|T|})_{N+|S|}}{(q^{\la+1+N-2|U|})_{M-|S|}}\\
&\quad\times
\prod_{i\in S,j\in T^c}\frac{\tha(w_i/w_j)}{\tha(qw_i/w_j)}
\prod_{i\in T^c,j\in [N]}\frac{\tha(qw_i/z_j)}{\tha(w_i/z_j)}
\prod_{i\in T\setminus S,j\in U}\frac{\tha(qw_i/z_j)}{\tha(w_i/z_j)}\\
&\quad\times\Phi(w_{T\setminus S};z_U;q^{\la+2+N-|U|});
\end{split}\end{multline*}
otherwise the left-hand side vanishes.
\end{corollary}

\section{Hypergeometric series}
\label{hss}

\subsection{Preliminaries on elliptic hypergeometric series}\label{mhs}

Elliptic $6j$-symbols can  be expressed in terms
of the elliptic hypergeometric series ${}_{12}V_{11}$, where, in general,
$${}_{m+5}V_{m+4}(a;b_1,\dots,b_m)=\sum_{y=0}^\infty\frac{\tha(aq^{2y})}{\tha(a)}\frac{(a,b_1,\dots,b_m)_y}{(q,aq/b_1,\dots,aq/b_m)_y}\,q^y. $$
For an introduction to  elliptic hypergeometric functions, the reader is referred to \cite[Chapter 11]{gr} or \cite{ss}. 
The   series arising from elliptic $6j$-symbols are terminating and balanced. Terminating means that $b_1=q^{-N}$, with  $N$ a non-negative integer, so that the summation is restricted to $0\leq y\leq N$. When $m=2k+1$, which is the only case of interest to us, balanced means that
$$b_1\dotsm b_{2k+1}=a^{k}q^{k-1}. $$ 
Two  important results for such series are the elliptic Bailey transformation
\begin{equation}\label{ebt}\begin{split}{}_{12}V_{11}(a;q^{-N},b,c,d,e,f,g)
&=\frac{(aq,aq/ef,\la q/e,\la q/f)_N}{(\la q,\la q/ef,aq/e,aq/f)_N}\\
&\quad\times{}_{12}V_{11}(\la;q^{-N},\la b/a,\la c/a,\la d/a,e,f,g),
 \end{split}\end{equation}
where $\la=qa^2/bcd$ and $a^3q^{N+2}=bcdefg$,
and the elliptic Jackson summation
$${}_{10}V_{9}(a;q^{-N},b,c,d,e)
=\frac{(aq,aq/bc,aq/bd,aq/cd)_N}{(aq/b,aq/c,aq/d,aq/bcd)_N},$$
where $a^2q^{N+1}=bcde$. These identities were first obtained by 
Frenkel and Turaev \cite{ft}, though  with some restriction on the parameters they are implicit in \cite{d}. 

We now turn to the multiple series defined by
\begin{multline}\label{ahs}V_{n}^m(a;b_1,\dots,b_{m+2};c_1,\dots,c_{m+n+2};z_1,\dots,z_n)\\
\begin{split}&= 
\sum_{y_1,\dots,y_n\geq 0}\frac{\Delta(zq^y)}{\Delta(z)}
q^{|y|}
\prod_{i=1}^n\frac{\tha(az_iq^{y_i+|y|})}{\tha(az_i)}\frac{\prod_{i=1}^n(az_i)_{|y|}\prod_{i=1}^{m+2}(b_i)_{|y|}}{\prod_{i=1}^{m+n+2}(aq/c_i)_{|y|}}\\
& \quad\times\prod_{i=1}^n\frac{\prod_{j=1}^{m+n+2}(c_jz_i)_{y_i}}{\prod_{j=1}^{n}(qz_i/z_j)_{y_i}\prod_{j=1}^{m+2}(aqz_i/b_j)_{y_i}},
\end{split}\end{multline}
where
\begin{equation}\label{evd}\Delta(z_1,\dots,z_n)=\prod_{1\leq j<k\leq n}z_j\theta(z_k/z_j),\end{equation}
$$\frac{\Delta(zq^y)}{\Delta(z)}=\prod_{1\leq j< k\leq n}\frac{q^{y_j}\theta(q^{y_k-y_j}z_k/z_j)}{\theta(z_k/z_j)}. $$
This type of series appears in \cite{rr,s1,s2} for $m=0$ and $m=1$ and in
 \cite{kan,rt} in general. 
Since
\eqref{evd} is an elliptic extension of the  $A_{n-1}$ Weyl denominator $\prod_{j<k}(z_j-z_k)$ (in fact, it is essentially the Macdonald denominator for the corresponding affine root system \cite{m}), these series are associated to root systems of type $A$. In the rational limit case, series of this type first appeared in the representation theory of unitary groups \cite{ajj,ccb,hbl}.

Note that $V_{n}^m$  does not change under the scaling $a\mapsto ta$, $c_j\mapsto tc_j$, $z_j\mapsto z_j/t$. 
This redundancy of notation is convenient but must be kept in mind. 
Note also that
\begin{multline*}V_{1}^m(a;b_1,\dots,b_{m+2};c_1,\dots,c_{m+3};z)\\
={}_{2m+10}V_{2m+9}(az;b_1,\dots,b_{m+2},c_1z,\dots,c_{m+3}z).
 \end{multline*}
In particular, $V_{1}^{1}={}_{12}V_{11}$ is the series related to
 elliptic $6j$-symbols. 

We call the series $V_n^m$ balanced when the parameters satisfy 
$$b_1\dotsm b_{m+2}c_1\dotsm c_{m+n+2}z_1\dotsm z_n=q^{m+1}a^{m+2}. $$
The series can terminate in different ways. For instance, when $b_1=q^{-N}$, the  summation is  restricted to  $|y|\leq N$. Alternatively, when $c_j=q^{-N_j}/z_j$ for $1\leq j\leq n$, the summation is restricted to $y_j\leq N_j$.

Kajihara and Noumi \cite{kan} and the present author \cite{rt} independently proved the transformation formula 
\begin{multline}\label{nkt}
\sum_{\substack{y_1,\dots,y_n\geq 0\\y_1+\dots+y_n=N}}
\frac{\De(zq^y)}{\De(z)}\prod_{k=1}^n\frac{\prod_{j=1}^{m+n}(a_jz_k)_{y_k}}
{\prod_{j=1}^m(w_jz_k)_{y_k}\prod_{j=1}^n(qz_k/z_j)_{y_k}}\\
=\sum_{\substack{y_1,\dots,y_m\geq 0\\y_1+\dots+y_m=N}}
\frac{\De(wq^y)}{\De(w)}
\prod_{k=1}^m\frac{\prod_{j=1}^{m+n}(w_k/a_j)_{y_k}}
{\prod_{j=1}^n(w_kz_j)_{y_k}\prod_{j=1}^m(qw_k/w_j)_{y_k}},
\end{multline}
where $w_1\dotsm w_m=z_1\dotsm z_na_1\dotsm a_{m+n}$.
This is a discrete analogue of an integral transformation of Rains \cite{ra}; the latter was recently given  a quantum field theory interpretation by Dolan and Osborn \cite{do}.

Eliminating one of the summation variables on each side of 
\eqref{nkt} yields a transformation between series of type
$V_{n-1}^{m-1}$ and $V_{m-1}^{n-1}$. Replacing $m$, $n$ by $m+1$ and $n+1$, and applying a standard argument of analytic continuation, the resulting identity takes the form 
\begin{multline}\label{rkt}
V_n^m \left(a;b,c,\frac{aq}{w_1},\dots,\frac{aq}{w_m};\frac{q^{-N_1}}{z_1},\dots,\frac{q^{-N_n}}{z_n},q^{M_1}w_1,\dots,q^{M_m}w_m,d,e;z_1,\dots,z_n\right)\\
=c^{|N|-|M|}\frac{(\la qd,\la qe)_{|M|}(aq/cd,aq/ce)_{|N|}}{(\la qd/c,\la qe/c)_{|M|}(aq/d,aq/e)_{|N|}}\\
\times\prod_{j=1}^m\frac{(\la qw_j/b,\la qw_j/c)_{M_j}}{(\la qw_j/bc,\la qw_j)_{M_j}}\prod_{j=1}^n\frac{(aqz_j/bc,a qz_j)_{N_j}}{(aqz_j/b,a qz_j/c)_{N_j}}
\, V_m^n\left(\la;b,c,\frac{\la q}{z_1},\dots,\frac{\la q}{z_n};\right.\\
\left.\frac{q^{-M_1}}{w_1},\dots,\frac{q^{-M_m}}{w_m},q^{N_1}z_1,\dots,q^{N_n}z_n,\frac 1d,\frac 1e;w_1,\dots,w_m\right),
\end{multline}
where $\la=bc/aq=q^{|N|-|M|}a/de$.

The case  $m=n=1$ of \eqref{rkt} is a ${}_{12}V_{11}$ transformation that
is different from \eqref{ebt}, but can be obtained as a consequence of that result.  When $m=0$,  the 
 function $V_{0}^n$ on the
right-hand side  of \eqref{rkt} should be interpreted as $1$, and we obtain 
 the multivariable elliptic Jackson summation
\cite[Cor.\ 5.3]{rr}
\begin{multline}\label{rjs}
V_n^0(a;b,c;{q^{-N_1}}/{z_1},\dots,{q^{-N_n}}/{z_n},d,e;z_1,\dots,z_n)\\
=c^{|N|}\frac{(aq/cd,aq/ce)_{|N|}}{(aq/d,aq/e)_{|N|}}\prod_{j=1}^n\frac{(aqz_j,a qz_j/bc)_{N_j}}{(aqz_j/b,a qz_j/c)_{N_j}},
\end{multline}
where $a^2q^{|N|+1}=bcde$.

Another multivariable elliptic Jackson summation is obtained in  \cite{rsm}; see \cite[Thm.\ 4.1]{sc} for the case $p=0$:
\begin{multline}\label{sjs}
\sum_{y_1,\dots,y_n=0}^{N_1,\dots, N_n}
\frac{\Delta(xq^y)}{\Delta(x)}\,q^{|y|}\frac{\tha(aq^{2|y|})}{\tha(a)}
\prod_{i=1}^n\frac{(aq^{1+|N|}/ex_i)_{|y|-y_i}(d/x_i)_{|y|}(ex_i)_{y_i}}{(d/x_i)_{|y|-y_i}(aq^{1+|N|-N_i}/ex_i)_{|y|}(aqx_i/d)_{y_i}}\\
\begin{split}&\quad\times\frac{(a,b,c)_{|y|}}{(aq^{1+|N|},aq/b,aq/c)_{|y|}}
\prod_{i,j=1}^n\frac{(q^{-N_j}x_i/x_j)_{y_i}}{(qx_i/x_j)_{y_i}}\\
&=\frac{(aq,aq/bc)_{|N|}}{(aq/b,aq/c)_{|N|}
}\prod_{i=1}^n\frac{(aqx_i/bd,aqx_i/cd)_{N_i}}{(aqx_i/d,aqx_i/bcd)_{N_i}},
\end{split}\end{multline}
where
$a^2q^{|N|+1}=bcde$.

\subsection{Hypergeometric series from generalized $6j$-symbols}

The  expressions  in  Theorems \ref{mcmt} and \ref{rat} are generalizations of  elliptic hypergeometric representations for  elliptic $6j$-symbols. 
In view of the discussion in \S \ref{dmps}, 
to recover the latter one should choose $S=[M-s+1,M]$, $T=[M-t+1,M]$,
$U=[N-u+1,N]$, $V=[N-v+1,N]$ and specialize $z_j=q^{j-1}\zeta$, $w_j=q^{j-1}\omega$. In Theorem \ref{mcmt}, the factor $\prod_{i\in S^c\setminus X,j\in X}\theta(qw_i/w_j)$ then vanishes unless $X=[1,x]$, while the factor $\prod_{i\in Y,j\in U\setminus Y}\theta(qz_i/z_j)$
vanishes unless $Y=[N-y+1,N]$. Since $x$ and $y$ are related by $y-x=s+u-M$, the expression reduces to a single sum.
Similar reductions  occur for the two expressions of Theorem  \ref{rat}. 
One can check that all three sums are of type $V_1^1={}_{12}V_{11}$, and that the equality of the three expressions follows from known transformation formulas for such series.  
This is the  case considered in \cite{kn}. We will now 
explain how to generalize these results to include series of type 
$V_n^m$.

First, we let
\begin{equation}\label{ws}w_j=q^{j-1}\om,\quad S=[M-s+1,M],\quad T=[M-t+1,M]\end{equation}
in  Theorem \ref{mcmt}. 
As above, we may write $X=[1,x]$. 
By Lemma \ref{pgl}, the elliptic weight functions factor, and we find that
$\mathcal R_{SU}^{TV}(\la;w;z)$ equals
\begin{multline*}\frac{(q)_{M-s}(q^{\la+2+M+N-2L})_{s}}{(q)_{t}(q^{\la+2+M-2t})_{t}(q^{\la+2+M+N-2L})_{|V|}}
\prod_{(i,j)\in(V\times V^c)\setminus(U\times U^c)}\frac{\tha(qz_i/z_j)}{\tha(z_i/z_j)}
\\
\begin{split}&\times\prod_{i\in U\cap V^c}\frac{\tha(q^{\la+1+M+N-L-|U|}\om/z_i)}{\tha(q^{M}\om/z_i)}\prod_{i\in U^c\cap V}\frac{\tha(q^{-\la-1+t}\om/z_i)}{\tha(q^{M}\om/z_i)}\\
&\times\sum_{\substack{Y\subseteq U\cap V,\\|Y|\geq L-M}}\frac{(q)_{L-|Y|}(q^{\la+2+N-|U|-|Y|})_{|Y|}}{(q)_{|Y|+M-L}(q^{-\la+2t-M})_{|V|-|Y|}}\prod_{i\in U\cap V\cap Y^c,j\in U^c\cap V^c}\frac{\tha(qz_i/z_j)}{\tha(z_i/z_j)}\\
&\times\prod_{i\in Y,j\in U\cap V\cap Y^c}\frac{\tha(qz_i/z_j)}{\tha(z_i/z_j)}
\prod_{i\in U^c\cap V^c}\frac{\tha(q^{|Y|+M-L}\om/z_i)}{\tha(q^{M}\om/z_i)}
\\
&\times\prod_{i\in Y}\frac{\tha(qz_i/\om)}{\tha(q^{1+L-M-|Y|}z_i/\om)}\prod_{i\in U\cap V\cap Y^c}\frac{\tha(q^{\la+1+M+N-L-|U|}\om/z_i,q^{-\la-1+t}\om/z_i)}{\tha(q^{|Y|+M-L}\om/z_i,q^M\om/z_i)}.
\end{split}
\end{multline*}

 Next, we specialize
\begin{subequations}\label{zs}
\begin{equation}z_{U\cap V}=(\eta_1,\dots,\eta_1q^{k_1-1},\dots,\eta_m,\dots,\eta_mq^{k_m-1}), \end{equation}
\begin{equation}z_{U^c\cap V^c}=(q^{1-l_1}\xi_1^{-1},\dots,\xi_1^{-1},\dots,
q^{1-l_n}\xi_n^{-1},\dots,\xi_n^{-1}
), \end{equation}
\end{subequations}
so that
\begin{equation}\label{pbc}N+|k|=|U|+|V|+|l|. \end{equation}
We stress that this is \emph{not} a restriction on the variables $z_i$, since the general case is included as $k_i\equiv l_i\equiv 1$.
Then, the only non-vanishing terms are those where
$$Y=[y_1+1,k_1]\cup\dots\cup[y_m+1,k_m], $$
with $0\leq y_i\leq k_i$, so that $|Y|=|k|-|y|$.
 One may check that
$$\prod_{i\in Y,j\in U\cap V\cap Y^c}\frac{\tha(qz_i/z_j)}{\tha(z_i/z_j)}
=(-1)^{|y|}q^{|y||k|-\binom{|y|}2}\frac{\Delta(\eta q^y)}{\Delta(\eta)}\prod_{i,j=1}^m\frac{(q^{-k_j}\eta_i/\eta_j)_{y_i}}{(q\eta_i/\eta_j)_{y_i}}, $$
cf.\  \cite{kan} or \cite[\S 7]{rr}.
After  simplification,  the  result takes the following form.

\begin{corollary}\label{ahc}
Assuming \eqref{ws} and \eqref{zs}, $\mathcal R_{SU}^{TV}(\la;w;z)$ equals
\begin{multline*}
\frac{(q)_{M-s}(q)_{L-|k|}}{(q)_{t}(q)_{M+|k|-L}}
\frac{(q^{\la+2+M+N-2L})_{s}(q^{\la+2+N-|U|-|k|})_{|k|}}{(q^{\la+2+M-2t})_{t}(q^{\la+2+M+N-2L})_{|V|}(q^{-\la+2t-M})_{|V|-|k|}}\\
\begin{split}&\times
\prod_{i\in U^c\cap V,\,j\in U\cap V^c}\frac{\tha(qz_i/z_j)}{\tha(z_i/z_j)}
\prod_{i\in U\cap V^c}\left(\frac{\tha(q^{\la+1+M+N-L-|U|}\om/z_i)}{\tha(q^M\om/z_i)}
\prod_{j=1}^m\frac{\tha(\eta_jq^{k_j}/z_i)}{\tha(\eta_j/z_i)}
\right)\\
&\times\prod_{i\in U^c\cap V}\left(\frac{\tha(q^{-\la+t-1}\om/z_i)}{\tha(q^M\om/z_i)}
\prod_{j=1}^n\frac{\tha(q^{l_j}\xi_jz_i)}{\tha(\xi_j z_i)}
\right)\prod_{i=1}^m\frac{(q\eta_i/\om)_{k_i}}{(q^{1+L-M-|k|}\eta_i/\om)_{k_i}}\\
&\times
\prod_{i=1}^n\frac{(q^{M+|k|-L}\om\xi_i)_{l_i}}{(q^{M}\om\xi_i)_{l_i}}\sum_{\substack{y_1,\dots,y_m\\0\leq y_i\leq k_i,\, |y|\leq |k|+M-L}}
\frac{\Delta(\eta q^y)}{\Delta(\eta)}q^{|y|}\\
&\times\prod_{i=1}^m\left(\frac{\tha(q^{L-M-|k|+|y|+y_i}\eta_i/\om)}{\tha(q^{L-M-|k|}\eta_i/\om)}\frac{(q^{L-M-|k|}\eta_i/\om)_{|y|}}{(q^{1+L-M-|k|+k_i}\eta_i/\om)_{|y|}}\right)\\
&\times\frac{(q^{L-M-|k|},q^{1+L-|k|})_{|y|}}{(q^{\la+2+N-|U|-|k|},q^{-\la+t+L-M-|k|})_{|y|}}\prod_{i=1}^n\frac{(q^{1+L-M-|k|}/\xi_i\om)_{|y|}}
{(q^{1+L-M-|k|-l_i}/\xi_i\om)_{|y|}}\\
&\times\prod_{i=1}^m\left(\frac{(q^{-\la-1+L+|U|-M-N}\eta_i/\om,q^{\la+1-t}\eta_i/\om)_{y_i}}{(q\eta_i/\om,q^{-M}\eta_i/\om)_{y_i}}\prod_{j=1}^m\frac{(q^{-k_j}\eta_i/\eta_j)_{y_i}}{(q\eta_i/\eta_j)_{y_i}}\prod_{j=1}^n\frac{(q^{l_j}\xi_j\eta_i)_{y_i}}{(\xi_j\eta_i)_{y_i}}\right).
\end{split}
\end{multline*}
\end{corollary}

In the notation \eqref{ahs}, the sum in Corollary \ref{ahc} can be written
\begin{multline*}V_{m}^n\left(\frac{q^{L-M-|k|}}\om;q^{L-M-|k|},q^{1+L-|k|},
\frac{q^{1+L-M-|k|}}{\om\xi_1},\dots,\frac{q^{1+L-M-|k|}}{\om\xi_n};\right.\\
\left.\frac{q^{\la+1-t}}{\om},\frac{q^{-\la-1+L+|U|-M-N}}{\om},\frac{q^{-k_1}}{\eta_1},\dots,\frac{q^{-k_m}}{\eta_m},q^{l_1}\xi_1,\dots,q^{l_n}\xi_n;
\eta_1,\dots,\eta_m\right). \end{multline*}
Note that, since $L=s+|U|=t+|V|$ and  \eqref{pbc} holds, the series is balanced.
Making the same specialization in the 
second expression of Theorem~\ref{rat}, one finds that 
$\mathcal R_{SU}^{TV}(\la;w;z)$ is an elementary factor times
\begin{multline*}V_{n}^m\left(q^{L-|k|}\om;q^{L-M-|k|},q^{1+L-|k|},
\frac{q^{1+L-|k|}\om}{\eta_1},\dots,\frac{q^{1+L-|k|}\om}{\eta_n};\right.\\
\left.q^{-\la-1+t}\om,q^{\la+1+M+N-L-|U|}\om,q^{k_1}\eta_1,\dots,q^{k_m}\eta_m,\frac{q^{-l_1}}{\xi_1},\dots,\frac{q^{-l_n}}{\xi_n};\xi_1,\dots,\xi_n\right). \end{multline*}
The fact that these two expressions agree is an instance of \eqref{rkt}. Thus, we have obtained an algebraic proof of this transformation.
(Although we only obtain \eqref{rkt} under an additional
discreteness condition on the parameters, that condition can be removed by analytic continuation similarly as in the proof of \cite[Cor.\ 5.3]{rr}.)

In \S \ref{bos}, it will be convenient to use the  expression obtained by replacing $y_i$ by $k_i-y_i$ in  Corollary \ref{ahc}. When $L\leq M$, the condition $|y|\leq |k|+M-L$
is trivially satisfied, and we find that $\mathcal R_{SU}^{TV}(\la;w;z)$ equals
\begin{multline}\label{iri}
q^{(s+N-M-|V|)|k|}\frac{(q)_{M-s}(q)_{L}}{(q)_{t}(q)_{M-L}}
\frac{(q^{\la+2+M+N-2L})_{s}}{(q^{\la+2+M-2t})_{t}(q^{\la+2+M+N-2L})_{|V|}(q^{-\la+2t-M})_{|V|}}\\
\begin{split}&\times
\prod_{i\in U^c\cap V,\,j\in U\cap V^c}\frac{\tha(qz_i/z_j)}{\tha(z_i/z_j)}
\prod_{i\in U\cap V^c}\left(\frac{\tha(q^{\la+1+M+N-L-|U|}\om/z_i)}{\tha(q^M\om/z_i)}
\prod_{j=1}^m\frac{\tha(\eta_jq^{k_j}/z_i)}{\tha(\eta_j/z_i)}
\right)\\
&\times\prod_{i\in U^c\cap V}\left(\frac{\tha(q^{-\la-1+t}\om/z_i)}{\tha(q^M\om/z_i)}
\prod_{j=1}^n\frac{\tha(q^{l_j}\xi_jz_i)}{\tha(\xi_j z_i)}
\right)\prod_{\substack{1\leq i\leq m,\\1\leq j\leq n}}\frac{(\eta_i\xi_j)_{k_i+l_j}}{(\eta_i\xi_j)_{k_i}(\eta_i\xi_j)_{l_j}}\\
&\times\prod_{i=1}^m\frac{(q^{-\la-1+L+|U|-M-N}\eta_i/\om,q^{\la+1-t}\eta_i/\om)_{k_i}}{(q^{L-M}\eta_i/\om,q^{-M}\eta_i/\om)_{k_i}}\prod_{i=1}^n\frac{(q^{M-L}\om\xi_i)_{l_i}}{(q^{M}\om\xi_i)_{l_i}}
\\
&\times
V_m^n(q^{M-L}\om;q^{\la+1+M-L-t},q^{-\la-1-N+|U|},q^{M-L+l_1}\xi_1\om,\dots,q^{M-L+l_n}\xi_n\om;\\
&\qquad\om,q^{M+1}\om,q/\xi_1,\dots,q/\xi_n,\eta_1,\dots,\eta_m;q^{-k_1}/\eta_1,\dots,q^{-k_m}/\eta_m).
\end{split}
\end{multline}
This expression remains valid for $L>M$, if interpreted as 
$$\frac{1}{(q)_{M-L}}\sum_{0\leq y_i\leq k_i}\frac{(\cdots)}{(q^{1+M-L})_{|y|}}=\sum_{0\leq y_i\leq k_i,\,|y|\geq L-M}\frac{(\cdots)}{(q)_{M-L+|y|}}. $$

\subsection{Biorthogonal functions}
\label{bos}

We have seen that, under appropriate specialization of the  parameters,  $\mathcal R_{SU}^{TV}$ can be written in terms
of the multiple  elliptic hypergeometric series $V_n^m$.
Making similar specializations in Proposition \ref{yop},
one obtains new results for such series.
We will only consider the unitarity relation \eqref{gbo}, and show that it leads to a system of biorthogonal functions of type $V_n^n$, generalizing the  functions of type $V_1^1={}_{12}V_{11}$  studied by Spiridonov and Zhedanov \cite{sz}.

The  functions that we will describe depend, apart from  $p$ and $q$, 
on $2n+3$ parameters $a,b,c,x_1,\dots,x_n,N_1,\dots,N_n$, with $N_i$  non-negative integers.  Fixing all these parameters, let
\begin{multline*}f_{u_1,\dots,u_n}(y_1,\dots,y_n)
=V_n^n\left(\frac ab;aq^{|y|},cq^{|u|},\frac{aq^{1+N_1}x_1}b,\dots,\frac {aq^{1+N_n}x_n}b;\right.\\
\left.1,\frac{aq^{1-|N|}}{b^2c},\frac{q^{-y_1}}{x_1},\dots,\frac{q^{-y_n}}{x_n},\frac{q^{-u_1}}{x_1},\dots,\frac{q^{-u_n}}{x_n};x_1,\dots,x_n\right),
 \end{multline*}
\begin{multline}\label{g}g_{u_1,\dots,u_n}(y_1,\dots,y_n)
=V_n^n\left(\frac{q^{-|N|}b}a;\frac{q^{-|y|-|N|}}a,\frac{q^{-|u|-|N|}}c,\frac{q^{1-|N|}b}{ax_1},\dots,\frac{q^{1-|N|}b}{ax_n};\right.\\
\left.q,\frac{q^{|N|}b^2c}{a},q^{y_1}x_1,\dots,q^{y_n}x_n,q^{u_1}x_1,\dots,q^{u_n}x_n;\frac{q^{-N_1}}{x_1},\dots,\frac{q^{-N_n}}{x_n}\right),\end{multline}
where it is assumed that $u_i$  are integers with $0\leq u_i\leq N_i$.

Note that, by \eqref{rkt}, $g_{u_1,\dots,u_n}(y_1,\dots,y_n)$ can alternatively be expressed as an
elementary factor times
\begin{multline*}
V_n^n\left(\frac{q^{-|u|-|y|-|N|-1}}{bc};\frac{q^{-|y|-N}}a,\frac{q^{-|u|-N}}c,\frac{q^{-|u|-|y|-|N|+N_1}x_1}{bc},\dots,\frac{q^{-|u|-|y|-|N|+N_n}x_n}{bc};\right.\\
\left.\frac{1}q,\frac{aq^{-|N|}}{b^2c},\frac{q^{-y_1}}{x_1},\dots,\frac{q^{-y_n}}{x_n},\frac{q^{-u_1}}{x_1},\dots,\frac{q^{-u_n}}{x_n};x_1,\dots,x_n\right).
\end{multline*}
Although it may seem more natural to normalize $g$ to be the latter $V_n^n$-series (without any prefactor), we prefer the definition \eqref{g} 
since it exhibits a simpler dependence on the variables $y_i$.

\begin{theorem}\label{bp}
The functions defined above satisfy the biorthogonality relations
\begin{equation}\label{bor}\sum_{y_1,\dots,y_n=0}^{N_1,\dots,N_n}w(y)f_u(y)g_v(y)=\delta_{u,v}\Gamma_u, \end{equation}
where
\begin{equation*}\begin{split}
w(y)&=\frac{\Delta(xq^y)}{\Delta(x)}\,q^{|y|}\frac{\tha (aq^{2|y|})}{\tha(a)}\frac{(a)_{|y|}}{(aq^{1+|N|})_{|y|}}\prod_{i,j=1}^n\frac{(q^{-N_j}x_i/x_j)_{y_i}}{(qx_i/x_j)_{y_i}}\\
&\quad\times\prod_{i=1}^n\frac{\tha(bq^{|y|-y_i}/x_i)(b/x_i)_{|y|}(q^{|N|}ax_i/b)_{y_i}}{\tha(b/x_i)(bq^{1-N_i}/x_i)_{|y|}(aqx_i/b)_{y_i}},
\end{split}\end{equation*}
\begin{equation*}\begin {split}
\Gamma_u&=c^{|N|}q^{|N|^2-|u|}\frac{\Delta(x)}{\Delta(xq^u)}\prod_{i,j=1}^n\frac{(qx_i/x_j)_{u_i}}{(q^{-N_j}x_i/x_j)_{u_i}}\frac{(aq)_{|N|}(q^{-|u|-|N|}/c)_{ |N|-|u|}(q^{1-2|u|}/c)_{|u|}}{(aq/b,bcq^{|N|})_{|N|}}\\
&\quad\times\prod_{i=1}^n\left(\frac{\tha(q^{-|u|}ax_i/bc)}{\tha(q^{-|u|+u_i}ax_i/bc)}\frac{(x_i,aq^{1-|N|}x_i/b^2c)_{N_i}}{(x_i/b,q^{-|u|}ax_i/bc)_{N_i}}\frac{(q^{u_i+|N|}ax_i/b)_{N_i-u_i}}
{(q^{1+u_i}ax_i/b)_{N_i-u_i}}\right).
\end{split}\end{equation*}
\end{theorem}

When $n=1$, the biorthogonal functions in Theorem \ref{bp} reduce to the one-variable functions of  Spiridonov and Zhedanov \cite{sz}.
We will briefly discuss   two  features that distinguish the one- and multivariable case. First of all, in self-explaining notation,
repeated use of \eqref{ebt} yields that
$g_u(y;a,b,c,x,N)$ equals an elementary prefactor times
$f_u(y;a,b,c,x/q,N)$. Thus, the system is ``almost'' orthogonal in the sense that  $f_u$ and $g_u$ are related by a parameter shift.
In the multivariable case, no analogous relation seems to exist.

Another peculiar property of the one-variable case is that 
$f_u$ and $g_u$ can be viewed as  \emph{rational} functions. To see this, note that  
$$f_u(y)=\sum_{k=0}^uC_k\frac{(aq^y,q^{-y})_k}{(q^{1-y}/b,aq^{y+1}/b)_k}, $$
with $C_k$ independent of $y$. It  follows from classical facts on elliptic functions 
 that $f_u$ is rational in the variable
$$\frac{\tha(s q^y,sq^{-y}/a)}{\tha(tq^y,tq^{-y}/a)}, $$
 with $s$ and $t$ arbitrary generic parameters. (Geometrically, identifying antipodal points on a complex torus gives the Riemann sphere.)
In the multivariable case, there seems to be no analogue of this rational parametrization.

Before explaining how Theorem \ref{bp} can be obtained from our findings above, we indicate a  direct proof. We will use an explicit matrix inversion found in \cite{rsm}; see \cite{sc} for the case $p=0$.
Namely, for $k,l,m$  multi-indices with $l_i\leq k_i\leq m_i$, $i=1,\dots, n$, let
$$A_{mk}(a,b)=\frac{(abq^{2|k|})_{|m|-|k|}\prod_{i=1}^n(aq^{|k|-k_i}/x_i)_{|m|-|k|}}{\prod_{i=1}^n(bx_iq^{1+k_i+|k|})_{m_i-k_i}\prod_{i,j=1}^n(q^{1+k_i-k_j}x_i/x_j)_{m_i-k_i}},$$
\begin{multline*}B_{kl}(a,b)=(-1)^{|k|-|l|}q^{\binom{|k|-|l|}2}\frac{\tha(abq^{2|l|})}{\tha(abq^{2|k|})}\prod_{i=1}^n\frac{\tha(aq^{|l|-l_i}/x_i)}{\tha(aq^{|k|-k_i}/x_i)}\\
\times\frac{(abq^{1+|l|+|k|})_{|k|-|l|}\prod_{i=1}^n(aq^{1+|l|-k_i}/x_i)_{|k|-|l|}}{\prod_{i=1}^n(bx_iq^{l_i+|k|})_{k_i-l_i}\prod_{i,j=1}^n(q^{1+l_i-l_j}x_i/x_j)_{k_i-l_i}}. \end{multline*}
Then, $B=A^{-1}$, that is, the equivalent identities 
$$\sum_{k}A_{mk}(a,b)B_{kl}(a,b)=\delta_{lm}=\sum_{k}B_{mk}(a,b)A_{kl}(a,b) $$
hold. In fact, the first relation is equivalent to the case $aq=ce$ of  \eqref{rjs}, while the second one is the  case $aq=bc$ of \eqref{sjs}.

Let $C_s$ be an arbitrary sequence, labelled by multi-indices $s$ such that $0\leq s_i\leq N_i$, $i=1,\dots, n$. Then,
\begin{multline*}\sum_{y}\sum_sC_sA_{us}(a,b)A_{ys}(c,d)\sum_t
\frac{1}{C_{N-t}}\,B_{N-t,v}(a,b)B_{N-t,y}(c,d)\\
=\sum_{st}\frac{C_s}{C_{N-t}}\,A_{us}(a,b)B_{N-t,v}(a,b)\delta_{s,N-t}
=\sum_{s}A_{us}(a,b)B_{sv}(a,b)=\delta_{uv}. 
\end{multline*}
A straight-forward computation reveals that
 Theorem \ref{bp} corresponds to the special case when
$(a,b,c,d)\mapsto(cb/a,a/b,b,a/b)$
and 
\begin{multline*}C_s=\frac{q^{2\sum_{i<j}s_is_j}}{±\prod_{i=1}^n x_i^{2s_i}}\left(\frac{b^2c}{qa}\right)^{|s|}\frac{\Delta(x)}{\Delta(xq^s)}\frac{(a,c)_{2|s|}}{(aq/b,q^{|N|}bc)_{|s|}}\\
\times\prod_{i=1}^n
\frac{(ax_i/b,aq^{1+N_i}x_i/b)_{|s|}(b/x_i,bc/ax_i)_{|s|-s_i}(x_i,aq^{1-|N|}x_i/b^2c)_{s_i}}{(ax_i/b,aqx_i/b)_{|s|+s_i} \prod_{j=1}^n(qx_i/x_j,q^{-N_j}x_i/x_j)_{s_i}}.
\end{multline*}
Clearly,  the same proof can be used to obtain more general, or different, biorthogonal systems. 

Finally, we explain how  Theorem \ref{bp} can be obtained from \eqref{gbo}.
Since the details of the computations are of little interest, we will be quite 
brief. First, we specialize $w$, $S$ and $T$ as in \eqref{ws}. We also assume that
\begin{align*}z&=(\zeta_1,\dots,\zeta_1q^{N_1-1},\dots,\zeta_n,\dots,\zeta_nq^{N_n-1}), \\
z_U&=(\zeta_1q^{N_1-u_1},\dots,\zeta_1q^{N_1-1},\dots,\zeta_nq^{N_n-u_n},\dots,\zeta_nq^{N_n-1}), \\
z_V&=(\zeta_1q^{N_1-v_1},\dots,\zeta_1q^{N_1-1},\dots,\zeta_nq^{N_n-v_n},\dots,\zeta_nq^{N_n-1}). \end{align*} 
Consider the symbol $\mathcal R^{XY}_{SU}(\la;w;z)$ in \eqref{gbo}. By Corollary \ref{qvs}, it vanishes identically unless 
$X=[M-x+1,M]$ and
$$z_Y=(\zeta_1q^{N_1-y_1},\dots,\zeta_1q^{N_1-1},\dots,\zeta_nq^{N_n-y_n},\dots,\zeta_nq^{N_n-1}). $$
This means that $\mathcal R^{XY}_{SU}$ can be expressed as in \eqref{iri}, where $M$, $N$, $L$, $s$, $\lambda$ and $\omega$ are unchanged, while the remaining parameters are replaced by $t\mapsto L-|y|$, $|U|\mapsto L-s$, $|V|\mapsto |y|$, $k_j\mapsto \min(u_j,y_j)$, $l_j\mapsto N_j-\max(u_j,y_j)$, $\eta_j\mapsto q^{N_j-\min(u_j,y_j)}\zeta_j$, $\xi_j\mapsto q^{1-N_j+\max(u_j,y_j)}\zeta_j^{-1}$.
As for the other generalized $6j$-symbol in  \eqref{gbo}, we first apply
Corollary \ref{cbsl} to write
$$\mathcal R^{VT}_{YX}(\la;z;w)=R^{T^cV^c}_{X^cY^c}(\la+M+N-2L;w^{-1};z^{-1}). $$
We can then express it as in \eqref{iri}, where $M$ and $N$ are unchanged, while $L\mapsto M+N-L$,
$s\mapsto M+|y|-L$, $t\mapsto M-t$, $|U|\mapsto N-|y|$, $|V|\mapsto N+t-L$,
$\la\mapsto\la+M+N-2L$, $\om\mapsto q^{1-M}\om$, $k_j\mapsto N_j-\max(v_j,y_j)$, $l_j\mapsto\min(v_j,y_j)$, $\eta_j\mapsto q^{1+\max(v_j,y_j)-N_j}$,
$\xi_j\mapsto q^{N_j-\min(v_j,y_j)}\zeta_j$. 
Inserting these explicit formulas in 
\eqref{gbo}, it reduces to \eqref{bor}, where 
$a=q^{\la+1+M-2L}$, $b=q^{\la+1-L}$, $c=q^{-\la-1-N}$ and $x_j=q^{-N_j}\om/\zeta_j$.
Since $M$ and $L$  are non-negative integers with $L\leq M+N$, we only obtain \eqref{bor} under additional discreteness conditions on the parameters.
However,  these conditions can be removed by analytic continuation, similarly as in the proof of \cite[Cor.\ 5.3]{rr}. In that sense, we have obtained an algebraic proof of Theorem \ref{bp}.

\section*{Appendix. Algebra symmetries.}
 
 \renewcommand{\thesection}{\Alph{section}}
 \setcounter{equation}{0}
\setcounter{theo+}{0}
\setcounter{section}{1}

The equations  \eqref{sun}, \eqref{phun}, \eqref{psu} and \eqref{ic}  reflect different forms of unitarity for symmetries of cobraided $\gh$-bialgebroids. 
We will describe how to encompass these in a general framework, which in particular simplifies the proof of Proposition \ref{ip}. It turns out that there are four types of unitarity, corresponding to a choice of direct or opposite product and coproduct. Moreover, one can incorporate twists by affine automorphisms of $\gha$ (e.g.\ the map $\la\mapsto -\la-2$ in Proposition \ref{ip}).

For  $A$  an  $\gh$-bialgebroid,  two opposite 
 $\gh$-bialgebroid structures $A^{\op}$ and $A^{\cop}$ 
on the complex vector space underlying $A$ were
introduced in \cite{kn}. We will  write  $A^{\coop}=(A^{\op})^{\cop}=(A^{\cop})^{\op}$. The bigradings on the opposite $\gh$-bialgebroids are given by 
$$A^{\op}_{\al\be}=A_{-\al,-\be},\qquad 
A^{\cop}_{\al\be}=A_{\be\al},\qquad 
A^{\coop}_{\al\be}=A_{-\be,-\al}.
 $$
The moment maps are given by
$$\mu_l^{A^{\op}}(f)x=x\mu_l^{A}(f),\qquad
\mu_l^{A^{\cop}}(f)x=\mu_r^{A}(f)x,\qquad
\mu_l^{A^{\coop}}(f)x=x\mu_r^{A}(f),
 $$
and the same equations with $l$ and $r$ interchanged.
The product on $A^{\cop} $ is the same as that on $A$, while $A^{\op}$ and
 $A^{\coop}$ are equipped with the opposite product $m^A\circ\sigma$.  
The coproduct on $A^{\op} $ is the same as that on $A$, while $A^{\cop}$ and
 $A^{\coop}$ have the opposite product $\sigma\circ \Delta^A$.  
The counit on  $A^{\cop} $ is the same as that on $A$, while $A^{\op}$ and
 $A^{\coop}$ have  counit $S^{\D}\circ \ep^A$.
Finally, if  $A$ is an $\gh$-Hopf algebroid with invertible antipode, then so are the opposite structures, with  antipode 
$S^{A^{\cop}}=S^A$, $S^{A^{\op}}=S^{A^{\coop}}=(S^A)^{-1}$.

Let $\chi^{\gha}$ be a linear automorphism of $\gha$ and $\chi^{\mg}$ a field automorphism of $\mg$ satisfying
\begin{equation}\label{cc}\chi^{\mg}\circ T_\al=T_{\chi^{\gha}(\al)}\circ\chi^{\mg}. \end{equation}
For instance, given an  invertible affine map  $\la\mapsto A\la+\la_0$ on $\gha$ one may define
\begin{equation}\label{caf}\chi^{\gha}(\la)= A^{-1}\la,\qquad \chi^{\mg}(f)(\la)=f(A\la+\la_0). \end{equation}
It follows from \eqref{cc} that
$$\chi^{\D}(fT_\al)=\chi^{\mg}(f)T_{\chi^{\gha}(\al)} $$
defines an algebra automorphism $\chi^{\D}$ of $\D$. From now on, we suppress the upper indices, denoting all three automorphisms by $\chi$. We will also write $$\chi^{\op}=S^{\D}\circ\chi^{\D}=\chi^{\D}\circ S^{\D}. $$

Next, we recall some rudiments of the duality theory for $\gh$-bialgebroids \cite[\S 3.1]{r}. 
It will be convenient to write 
$\{x,\xi\}=\xi(x)$,
where $\xi$ is a $\C$-linear map from an  $\gh$-bialgebroid  $A$ to $\D$. Let $A'$  be the space of such maps $\xi$ such that
$$\{\mu_l(f)x,\xi\}=f\circ\{x,\xi\},\qquad \{x\mu_r(f),\xi\}=\{x,\xi\}\circ f.$$
It is an associative algebra with product
$$\{x,\xi\eta\}=\sum_{(x)}\{x',\xi\}T_{\om_{12}(x)}\{x'',\eta\} $$
and unit element $\ep$.

\begin{lemma}\label{dl}
Fix $\chi$ as above, and let $A$ and $B$ be two $\gh$-bialgebroids. 
Let $\phi:A\rightarrow B$ be a $\mathbb C$-linear map such that
$\phi(A_{\al\be})\subseteq B_{\chi^{-1}(\al),\chi^{-1}(\beta)}$,
$$\phi(\mu_l^A(f)x)=\mu_l^B(\chi^{-1}(f))\phi(x),\qquad
\phi(\mu_r^A(f)x)=\mu_r^B(\chi^{-1}(f))\phi(x).
 $$
Then,
$\{x,\phi'(\xi)\}=\chi(\{\phi(x),\xi\}) $
defines a map $\phi':\,B'\rightarrow A'$. Moreover, if 
$$(\phi\ot\phi)\circ\Delta^A=\Delta^B\circ\phi,\qquad
\chi\circ\ep^B\circ\phi=\ep^A,
 $$
then $\phi'$ is an algebra homomorphism.
\end{lemma}

The proof  is straight-forward.

\begin{lemma}
Let $A$ be an $\gh$-bialgebroid equipped with a cobraiding. For  
$s\in\{\emptyset,\op,\cop,\coop\}$, there is an algebra homomorphism $i^s:\,A^s\rightarrow (A^s)'$ given by
\begin{align*}\{y,i(x)\}&=\lx x,y\rx, \\
\{y,i^{\op}(x)\}&=S^{\D}(\lx y,x\rx), \\
\{y,i^{\cop}(x)\}&=\lx y,x\rx, \\
\{y,i^{\coop}(x)\}&=S^{\D}(\lx x,y\rx). \end{align*}
\end{lemma}

Again, the proof is straight-forward.

\begin{definition}\label{udef}
Let $A$ be an $\gh$-bialgebroid equipped with a cobraiding. Fix $\chi$ as above, and let $s\in\{\emptyset,\op,\cop,\coop\}$. Then, a map $\phi:\,A\rightarrow A$ is called  $(\chi,s)$-\emph{unitary} if, when viewed as 
a map $A\rightarrow A^s$, it is
 an algebra homomorphism, satisfies all conditions of \emph{Lemma \ref{dl}}, and
\begin{equation}\label{uc}\phi'\circ i^s\circ\phi=i.\end{equation}
\end{definition}

\begin{subequations}\label{uce}
More explicitly,  $\phi$ is 
$(\chi,\id)$-unitary if
it is an algebra homomorphism and
$$\phi(\mu_l(f))=\mu_l(\chi^{-1}(f)),\qquad \phi(\mu_r(f))=\mu_r(\chi^{-1}(f)),  $$
$$\phi(A_{\al\be})\subseteq A_{\chi^{-1}(\al),\chi^{-1}(\be)}, $$
$$(\phi\ot \phi)\circ\De=\De\circ\phi,\qquad \chi\circ\ep\circ\phi=\ep, $$
\begin{equation}\lx x,y\rx=\chi(\lx\phi(x),\phi(y)\rx); \end{equation}
it is $(\chi,\op)$-unitary if
it is an algebra antihomomorphism and
$$\phi(\mu_l(f))=\mu_l(\chi^{-1}(f)),\qquad \phi(\mu_r(f))=\mu_r(\chi^{-1}(f)),  $$
$$\phi(A_{\al\be})\subseteq A_{-\chi^{-1}(\al),-\chi^{-1}(\be)}, $$
$$(\phi\ot \phi)\circ\De=\De\circ\phi,\qquad  \chi^{\op}\circ\ep\circ\phi=\ep, $$
\begin{equation}\lx x,y\rx=\chi^{\op}(\lx\phi(y),\phi(x)\rx); \end{equation}
it is  $(\chi,\cop)$-unitary if
it is an algebra homomorphism and
$$\phi(\mu_l(f))=\mu_r(\chi^{-1}(f)),\qquad \phi(\mu_r(f))=\mu_l(\chi^{-1}(f)),  $$
$$\phi(A_{\al\be})\subseteq A_{\chi^{-1}(\be),\chi^{-1}(\al)}, $$
$$\sigma\circ(\phi\ot \phi)\circ\De=\De\circ\phi,\qquad  \chi\circ\ep\circ\phi=\ep, $$
\begin{equation}\lx x,y\rx=\chi(\lx\phi(y),\phi(x)\rx); \end{equation}
and, finally,  $\phi$ is $(\chi,\coop)$-unitary if
it is an algebra antihomomorphism and
$$\phi(\mu_l(f))=\mu_r(\chi^{-1}(f)),\qquad \phi(\mu_r(f))=\mu_l(\chi^{-1}(f)),  $$
$$\phi(A_{\al\be})\subseteq A_{-\chi^{-1}(\be),-\chi^{-1}(\al)}, $$
$$\sigma\circ(\phi\ot \phi)\circ\De=\De\circ\phi,\qquad  \chi^{\op}\circ\ep\circ\phi=\ep, $$
\begin{equation}\lx x,y\rx=\chi^{\op}(\lx\phi(x),\phi(y)\rx). \end{equation}
\end{subequations}

If $A$ is equipped with an invertible antipode, then it is natural to require
\begin{equation}\label{hua}\phi\circ S=S^{A^s}\circ\phi, \end{equation}
that is, $\phi\circ S=S\circ\phi$ for $s\in\{\emptyset,\cop\}$
and $\phi\circ S=S^{-1}\circ\phi$ for $s\in\{\op,\coop\}$.

Since  \eqref{uc} is an equality between compositions of algebra homomorphisms, it is natural in the sense that if the two sides agree on two elements $x$ and $y$, they agree on $xy$. We also need a dual version of this naturality. To this end, we observe that if $\phi$ satisfies all conditions of Lemma \ref{dl}, then this is also true when  $\phi$ is viewed as a map from 
$A^{\cop}$ to $B^{\cop}$. It is then easy to check that \eqref{uc} is equivalent to
$$\phi'\circ (i^{s})^{\cop}\circ\phi=i^{\cop},$$
when  
$\phi$ is viewed as a map  $A^{\cop}\rightarrow (A^{s})^{\cop}$.
Together, the naturality properties of these two versions of \eqref{uc}
mean that, assuming the other conditions of $\phi$, if one of the equalities \eqref{uce} hold with $(x,y)$ replaced by $(x_1,y)$ and $(x_2,y)$ it holds also for $(x_1x_2,y)$, and if it holds for  $(x,y_1)$ and $(x,y_2)$, it holds also for $(x,y_1y_2)$. Thus, it is enough to check unitarity on a set of generators. 
Having made this observation, the proof of  Proposition \ref{ip}
is reduced to straight-forward verification.

If we let
 $\chi_0=\psi^{\D}$ denote the automorphism constructed as in \eqref{caf} from the affine map $\lambda\mapsto -\la-2$ of $\gha=\mathbb C$, then we 
can give examples of  eight types of  $(\chi,s)$-unitary maps according to the following table (in each case, the additional axiom \eqref{hua} is valid):
\begin{figure}[h]
\begin{center}
\begin{tabular}{|c|c|c|c|c|c|c|c|c|}\hline
&$\id$&$\ast$&$\ast\circ S$&$S$&$\psi$&$\phi$&$\phi\circ S$&$\psi\circ S$\\
\hline
$\chi$&$\id$&$\id$&$\id$&$\id$&$\chi_0$&$\chi_0$&$\chi_0$&$\chi_0$\\
\hline
$s$&$\id$&$\op$&$\cop$&$\coop$&$\id$&$\op$&$\cop$&$\coop$\\
\hline
\end{tabular}\ .
\end{center}
\end{figure}

 \end{document}